\documentclass[11pt,english]{article}
\usepackage{geometry}
\usepackage{amsthm}
\usepackage{amsmath}
\usepackage{amssymb}
\usepackage{bbm}
\usepackage[nottoc]{tocbibind}
\usepackage{todonotes}
\usepackage{array}
\usepackage{makecell}
\usepackage{threeparttable}
\usepackage{faktor}
\usepackage{tikz}
\usepackage{tikz-cd}
\usepackage{mathtools}
\usepackage{marginnote}
\usepackage[unicode=true,pdfusetitle,
 bookmarks=false,
 breaklinks=false,pdfborder={0 0 1},backref=page,colorlinks=true]{hyperref}
\hypersetup{
	colorlinks=true,
	linkcolor=blue,
	filecolor=magenta,      
	urlcolor=cyan,
	citecolor=blue
}

\usepackage{nameref}
\usepackage{aliascnt}
\usepackage{wrapfig}
\usepackage{multicol}
\usepackage{enumitem}
\usepackage{ifthen}
\usepackage{bm}
\usepackage{relsize}
\usepackage{color,soul}
\usepackage{pdfpages}

\let\oldforall\forall
\let\forall\undefined
\DeclareMathOperator{\forall}{\oldforall}

\let\oldexists\exists
\let\exists\undefined
\DeclareMathOperator{\exists}{\oldexists}

\newcommand\capmystring[1]{\capmystringaux#1\relax}
\def\capmystringaux#1#2\relax{\uppercase{#1}\lowercase{#2}}

\newcommand{\registerTheoremType}[1]{
	\newaliascnt{#1}{theorem}  
	\newtheorem{#1}[#1]{\capmystring{#1}}  
	\aliascntresetthe{#1}
	\expandafter\providecommand\csname #1autorefname\endcsname{\capmystring{#1}} %
}

\theoremstyle{plain}
\newtheorem{theorem}{Theorem}[section]
\registerTheoremType{lemma}
\registerTheoremType{corollary}
\registerTheoremType{remark}
\registerTheoremType{proposition}

\theoremstyle{definition}
\registerTheoremType{definition}
\registerTheoremType{example}
\registerTheoremType{notation}
\registerTheoremType{question}

\numberwithin{equation}{section}

\newcolumntype{C}{>{$}c<{$}} 

\DeclareMathOperator{\Ima}{Im}
\DeclareMathOperator{\sgn}{sgn}

\newcommand{\define}{\mathrel{\mathop:}=}
\newcommand{\field}[1]{\mathbb{F}_{#1}}
\newcommand{\fieldInvertible}[1]{\field{#1}^\times}
\newcommand{\trace}[2]{\operatorname{Tr}_{{#1}/{#2}}}

\newcommand{\complex}{\mathbb{C}}

\newcommand{\factor}[2]{\faktor{#1}{#2}}

\makeatletter
\DeclareRobustCommand*{\mfaktor}[3][]
{
	{ \mathpalette{\mfaktor@impl@}{{#1}{#2}{#3}} }
}
\newcommand*{\mfaktor@impl@}[2]{\mfaktor@impl#1#2}
\newcommand*{\mfaktor@impl}[4]{
	\settoheight{\faktor@zaehlerhoehe}{\ensuremath{#1#2{#3}}}%
	\settoheight{\faktor@nennerhoehe}{\ensuremath{#1#2{#4}}}%
	\raisebox{-0.5\faktor@zaehlerhoehe}{\ensuremath{#1#2{#3}}}%
	\mkern-4mu\diagdown\mkern-5mu%
	\raisebox{0.5\faktor@nennerhoehe}{\ensuremath{#1#2{#4}}}%
}
\makeatother

\newcommand{\rfactor}[2]{\mfaktor{#2}{#1}}

\newcommand{\linearGroupField}[1]{\field{#1}}

\newcommand{\SP}[2]{Sp_{#1}({\linearGroupField{#2}})}

\newcommand{\XL}[3]{#1L_{#2}(\linearGroupField{#3})}
\newcommand{\SL}[2]{\XL{S}{#1}{#2}}
\newcommand{\GL}[2]{\XL{G}{#1}{#2}}
\newcommand{\GammaL}[2]{\XL{\Gamma}{#1}{#2}}
\newcommand{\GammaLExt}[3]{\Gamma L_{#1}\left(\factor{\linearGroupField{#2}}{\linearGroupField{#3}}\right)}

\newcommand{\aut}[1]{\operatorname{Aut}\left(#1\right)}
\newcommand{\gal}[2]{\operatorname{Gal}\left(\factor{#1}{#2}\right)}

\newcommand{\projective}{\mathbb{P}}
\newcommand{\projectiveVector}[1]{\overline{#1}}

\newcommand{\projectiveSemiPairs}[1]{\Phi(#1)}

\newcommand{\PXL}[3]{\projective\XL{#1}{#2}{#3}}
\newcommand{\PGammaL}[2]{\PXL{\Gamma}{#1}{#2}}
\newcommand{\PGL}[2]{\PXL{G}{#1}{#2}}
\newcommand{\PSL}[2]{\PXL{S}{#1}{#2}}

\newcommand{\AGL}[2]{A\GL{#1}{#2}}
\newcommand{\AGammaL}[2]{A\GammaL{#1}{#2}}

\newcommand{\cardinality}[1]{\left\lvert#1\right\rvert}

\newcommand{\classicalFamily}[2]{\classicalFamilySymbol{#1}(#2)}
\newcommand{\classicalFamilySymbol}[1]{\mathcal{C}_{#1}}

\newcommand{\classicalExceptions}{\mathcal{S}}

\newcommand{\characters}[1]{\operatorname{Ch}\left({#1}\right)}
\newcommand{\irreducibleCharacters}[1]{\widehat{#1}}
\newcommand{\innerProduct}[2]{\left< #1, #2 \right>}
\newcommand{\trivialCharacter}[1]{\mathbbm{1}_{#1}}
\newcommand{\characterSum}[2]{\innerProduct{{#1}\vert_{#2}}{\trivialCharacter{#2}}_{#2}}
\newcommand{\linearSpan}[1]{Span({#1})}
\newcommand{\groupSpan}[1]{\left< #1 \right>}

\newcommand{\induced}[3]{\operatorname{Ind}_{#2}^{#3}(#1)}
\newcommand{\inducedTrivial}[2]{\induced{\trivialCharacter{#1}}{#1}{#2}}

\newcommand{\restricted}[2]{{#1}\vert_{#2}}
\newcommand{\hcLikePullback}[3]{\operatorname{P}_{#2}^{#3}(#1)}
\newcommand{\hcLikeInduction}[3]{\operatorname{R}_{#2}^{#3}(#1)}

\newcommand{\pullbackChiByF}[2]{\pullbackSymbol{#2}#1}
\newcommand{\pullbackSymbol}[1]{#1^{*}}

\newcommand{\stab}[2]{{#1}_{#2}}
\newcommand{\orbit}[2]{\operatorname{O}_{#1}({#2})}
\newcommand{\fixXg}[2]{#1^#2}

\newcommand{\normalSubgroup}{\trianglelefteq}
\newcommand{\subgroup}{\le}
\newcommand{\properSubgroup}{\lneqq}
\newcommand{\supergroup}{\ge}

\newcommand{\normalizer}[2]{\operatorname{N}_{#2}({#1})}

\newcommand{\unit}{e}
\newcommand{\groupCenter}[1]{Z(#1)}

\newcommand{\symmetric}[1]{S_{#1}}
\newcommand{\alternating}[1]{A_{#1}}
\newcommand{\cyclic}[1]{C_{#1}}

\newcommand{\symmetries}[1]{\operatorname{Sym}\left(#1\right)}

\newcommand{\subgroupIndex}[2]{\left[ #1 \colon #2 \right]}

\newcommand{\acts}{\circlearrowleft}
\newcommand{\divides}{\mathlarger{\mathlarger{\mid}}}

\renewcommand{\epsilon}{\varepsilon}

\newcommand{\defineNotation}[1]{#1\marginnote{$#1$}}
\newcommand{\defineTextNotation}[1]{#1\marginnote{#1}}

\newcommand{\defineNotationsSilent}[1]{\marginnote{$#1$}}

\newcommand{\textChiVanishesOnK}[2]{$\chiKDoesntContainTrivial{#1}{#2}$}
\newcommand{\textChiDoesntVanishOnK}[2]{$\chiKContainsTrivial{#1}{#2}$}

\newcommand{\chiKContainsTrivial}[2]{\chiContainsTrivial{\restricted{#1}{#2}}{#2}}
\newcommand{\chiKDoesntContainTrivial}[2]{\chiDoesntContainTrivial{\restricted{#1}{#2}}{#2}}

\newcommand{\chiContainsTrivial}[2]{\phiContainsPsi{#1}{\trivialCharacter{#2}}}
\newcommand{\chiDoesntContainTrivial}[2]{\phiDoesntContainPsi{#1}{\trivialCharacter{#2}}}

\newcommand{\phiContainsPsi}[2]{#1 \supseteq #2}
\newcommand{\phiDoesntContainPsi}[2]{#1 \not \supseteq #2}

\newcommand{\specialRepresentationsSymbol}{\mathcal{R}}
\newcommand{\specialCharactersSymbol}{\mathcal{X}}
\newcommand{\specialCharacterSet}[1]{\specialCharactersSymbol\left(#1\right)}
\newcommand{\specialCharacterSetGLnq}[2]{\specialCharacterSet{\GL{#1}{#2}}}
\newcommand{\specialCharacterSetSLnq}[2]{\specialCharacterSet{\SL{#1}{#2}}}

\newcommand{\specialCharacterSetXLnq}[1]{\specialCharacterSet{#1}}

\newcommand{\specialMinCharacterSetCardinality}[1]{\operatorname{c}_{\min}\left(#1\right)}

\newcommand{\textCMarksProperSubgroups}[2]{#1 marks every proper subgroup of #2}
\newcommand{\textCMarksAllProperSubgroups}[2]{\textCMarksProperSubgroups{#1}{#2}}
\newcommand{\textCMarksSubgroups}[2]{#1 marks every subgroup of #2}
\newcommand{\textCMarksSubgroupsThatDoNotContainH}[3]{\textCMarksSubgroups{#1}{#2} that does not contain #3}
\newcommand{\textCMarksSubgroupsThatAreNotxkTransitive}[4]{\textCMarksSubgroups{#1}{#2} that is not \xkGTransitive{#3}{#4}{#2}}
\newcommand{\textCMarksSubgroupsThatAreNotkTransitive}[3]{\textCMarksSubgroups{#1}{#2} that is not #3-transitive}

\newcommand{\textCMarkSubgroups}[2]{#1 mark every subgroup of #2}
\newcommand{\textCMarkSubgroupsThatDoNotContainH}[3]{\textCMarkSubgroups{#1}{#2} that does not contain #3}

\newcommand{\cuspidalForGLTwo}{\phi_1}
\newcommand{\transitivityCharacter}{\theta_1}
\newcommand{\twoTransitivityCharacter}{\theta_2}
\newcommand{\transitivityCharacters}{\theta_i}
\newcommand{\unipotentCharacter}[1]{\chi^{#1}}
\newcommand{\additionalForGLTwo}{\psi_1}

\newcommand{\snCharacter}[1]{\chi^{#1}}

\newcommand{\xKOrK}[2]{\ifthenelse{\equal{#1}{}}{#2}{{#1}_{#2}}}
\newcommand{\xOneToXk}[2]{\ifthenelse{\equal{#2}{1}}{\xKOrK{#1}{1}}{\ifthenelse{\equal{#2}{2}}{\xKOrK{#1}{1}, \xKOrK{#1}{2}}{\ifthenelse{\equal{#2}{3}}{\xKOrK{#1}{1}, \xKOrK{#1}{2}, \xKOrK{#1}{3}}{\xKOrK{#1}{1}, \dots, \xKOrK{#1}{#2}}}}}

\newcommand{\xGTransitive}[2]{#1-#2-transitive}
\newcommand{\xkGTransitive}[3]{\xGTransitive{$(\xOneToXk{#1}{#2})$}{#3}}
\newcommand{\xGTransitivity}[2]{#1-#2-transitivity}
\newcommand{\xkGTransitivity}[3]{\xGTransitivity{$(\xOneToXk{#1}{#2})$}{#3}}

\newcommand{\declareFootnote}[2]{\footnote{\label{footnote:#1}#2}}

\newcommand{\integers}{\mathbb{Z}}
\newcommand{\positiveIntegers}{\integers_{\ge 1}}

\newcommand{\cayleyGraph}[2]{\operatorname{Cay}(#1, #2)}
\newcommand{\regularRepresentation}{\rho_{Reg}}

\newcommand{\mathSuchThat}{\text{ s.t. }}

\newcommand{\declareAbcCounter}[1]{%
	\newcounter{#1}
	\expandafter\def\csname the#1\endcsname{(\alph{#1})}%
}
\newcommand{\abcCounter}[1]{\refstepcounter{#1}\expandafter\csname the#1\endcsname}

\title{A Generation Criterion for Subsets of $SL_n(\mathbb{F}_q)$}
\author{Ziv Greenhut \\ Tel Aviv University}
\begin{document}
	\maketitle
	\begin{abstract}
		Let $G_0$ be a either $\SL{n}{q}$, the special linear group over the finite field with $q$ elements, or $\PSL{n}{q}$, its projective quotient, and let $\Sigma$ be a symmetric subset of $G_0$, namely, if $x \in \Sigma$ then $x^{-1} \in \Sigma$. We find a certain set $\specialRepresentationsSymbol(G_0)$ of irreducible representations of $G_0$ whose size is at most $5$, such that $\Sigma$ generates $G_0$ if and only if $\cardinality{\Sigma}$ is not an eigenvalue of ${\sum_{\sigma \in \Sigma} \rho(\sigma)}$ for every $\rho \in \specialRepresentationsSymbol(G_0)$.
		
		To achieve this result, let $G$ be either $\GL{n}{q}$ or $\PGL{n}{q}$. We consider $\specialCharacterSet{G}$, some set of irreducible nontrivial characters of $G$, whose size is at most $5$. We show that for every subgroup $K \subgroup G$ that does not contain $G_0$, the restriction to $K$ of at least one of the characters in $\specialCharacterSet{G}$ contains the trivial character as an irreducible summand. We then restrict the characters to $G_0$ and use standard arguments about the Cayley graph of $G$ to imply the result. In addition, we obtain slightly weaker results about the generation of symmetric subsets of $G$.
		
		We finish by considering $\symmetric{n}$, the symmetric group on $n$ elements, and presenting $\specialRepresentationsSymbol(\symmetric{n})$, a set of eight irreducible nontrivial representations of $\symmetric{n}$, such that a symmetric subset $\Sigma \subseteq \symmetric{n}$ generates $\symmetric{n}$ if and only if $\cardinality{\Sigma}$ is not an eigenvalue of ${\sum_{\sigma \in \Sigma} \rho(\sigma)}$ for every $\rho \in \specialRepresentationsSymbol(\symmetric{n})$, which is an improvement upon the previously known set of $12$ irreducible nontrivial representations of $\symmetric{n}$ that satisfies this condition.
	\end{abstract}

	\tableofcontents
	
	\section{Introduction} \label{section:introduction}
	\paragraph{ } Let $G$ be a finite group, and let $\Sigma$ be a symmetric subset of $G$, namely, if $x \in \Sigma$ then $x^{-1} \in \Sigma$. We look for criteria to determine whether $\Sigma$ generates $G$. 
	
	For a representation $\rho$ of $G$, define
	\[\defineNotation{A_\rho(\Sigma)} \define \sum_{s \in \Sigma}\rho(s).\]
	Then, the following proposition presents one such criterion.
	
	\begin{proposition} \label{proposition:characters_eigenvalues_generation}
		Let $G$ be a finite group, and let $\Sigma$ be a symmetric subset of $G$. Then, $\Sigma$ generates $G$ if and only if for every nontrivial irreducible representation $\rho$ of $G$, the maximal eigenvalue of $A_\rho(\Sigma)$ is strictly less than $\cardinality{\Sigma}$.
	\end{proposition}
		
	We prove this proposition using Cayley graphs.	Before we prove the proposition, we define and state several facts about Cayley graphs. 
	
	Define $\cayleyGraph{G}{\Sigma}$, the Cayley graph of $G$ with respect to $\Sigma$, as the $\cardinality{\Sigma}$-regular graph whose vertices are the elements of $G$, and edges are $(g, gs)$ for all $g \in G$ and $s \in \Sigma$ (with a single edge between $g$ and $gs$ if $s^2 = 1$). By the definition of the Cayley graph, $\Sigma$ generates $G$ if and only if $\cayleyGraph{G}{\Sigma}$ is connected.
	
	For a graph $\Gamma$ with $n$ vertices, let 
	\[\lambda_1(\Gamma) \ge \lambda_2(\Gamma) \ge \dots \lambda_n(\Gamma)\] 
	be the eigenvalues of the adjacency matrix of the graph. The graph is connected if and only if $\lambda_1(\Gamma)$ has multiplicity $1$, namely, if ${\lambda_1(\Gamma) \gneqq \lambda_2(\Gamma)}$. In this situation, we say that the graph admits a spectral gap.
	
	Applying it to the case of Cayley graphs, we deduce that the graph $\cayleyGraph{G}{\Sigma}$ admits a spectral gap if and only if $\Sigma$ generates $G$. Therefore, we consider the adjacency matrix of the Cayley graph, and study its spectrum.
	
	The adjacency matrix of $\cayleyGraph{G}{\Sigma}$ is equal to $A_{\regularRepresentation}(\Sigma)$, where $\regularRepresentation$ is the right regular representation of $G$, that is, the permutation representation of the action of $G$ on itself, defined by right multiplication.
	
	A known fact in representation theory\declareFootnote{representation_theory_reference}{Throughout this paper we use standard facts about representation theory. \cite{representation_theory_fulton_harris} is a good reference for this subject.} states that the regular representation $\regularRepresentation$ decomposes to the sum of all of the irreducible representations of $G$, each with multiplicity equal to its dimension. Therefore, the eigenvalues of $A_{\regularRepresentation}(\Sigma)$ are (up to multiplicities) the union of the eigenvalues of $A_\rho(\Sigma)$, for every irreducible representation $\rho$ of $G$. As such, every eigenvalue of $A_{\regularRepresentation}(\Sigma)$ can be associated with some irreducible representation $\rho$ of $G$. For example, the eigenvalue associated with the trivial representation is ${\lambda_1(\cayleyGraph{G}{\Sigma}) = \cardinality{\Sigma}}$.
	
	We can now prove the proposition.
	
	\begin{proof}[Proof of \autoref{proposition:characters_eigenvalues_generation}]
		The set $\Sigma$ generates $G$ if and only if the Cayley graph $\cayleyGraph{G}{\Sigma}$ is connected. This graph is connected if and only if it admits a spectral gap, namely that $\lambda_2(\cayleyGraph{G}{\Sigma}) \lneqq \cardinality{\Sigma}$.
		
		The eigenvalues of $A_{\regularRepresentation}(\Sigma)$ are (up to multiplicities) the union of the eigenvalues of $A_\rho(\Sigma)$, for every irreducible representation $\rho$ of $G$. This implies that $\Sigma$ generates $G$ if and only if for every nontrivial irreducible representation $\rho$ of $G$, the maximal eigenvalue of $A_\rho(\Sigma)$ is strictly less than $\cardinality{\Sigma}$.
	\end{proof}
	
	When $n \in \positiveIntegers$ is large enough, consider the symmetric group $\symmetric{n}$. It is known that for several types of sets $\Sigma$, there exists some set $\specialRepresentationsSymbol$ of irreducible representations of $\symmetric{n}$ with $\cardinality{\specialRepresentationsSymbol} \le 8$, such that the second largest eigenvalue of the Cayley graph is always associated with some representation in $\specialRepresentationsSymbol$. Then, it is much easier to determine the spectral gap ${\lambda_1(\cayleyGraph{\symmetric{n}}{\Sigma}) - \lambda_2(\cayleyGraph{\symmetric{n}}{\Sigma})}$. In particular, it is much easier to determine whether $\cayleyGraph{\symmetric{n}}{\Sigma}$ admits a spectral gap, as the matrices of these representations are of a dimension much smaller than $n! = \cardinality{\symmetric{n}}$, the dimension of the regular representation. See \cite{aldous_proof} and \cite{aldous_proof_normal_sets}.
	
	Given a finite group $G$, this problem can be naturally generalized to the problem of finding a relatively small set $\specialRepresentationsSymbol$ of irreducible representations of $G$, such that the second largest eigenvalue of $A_{\regularRepresentation}(\Sigma)$ is associated with some representation in $\specialRepresentationsSymbol$. Then, the second largest eigenvalue of $A_{\regularRepresentation}(\Sigma)$ is the largest eigenvalue of one of the matrices $\{A_\rho(\Sigma) \colon \rho \in \specialRepresentationsSymbol\}$.
	
	In this paper, we consider the following relaxation of this generalized problem. Given some finite group $G$, we wish to find some relatively small set $\specialRepresentationsSymbol$ of irreducible representations of $G$, such that for every symmetric subset $\Sigma$, $\cayleyGraph{G}{\Sigma}$ admits a spectral gap (equivalently, $\Sigma$ generates $G$) if and only if the eigenvalues of $A_\rho(\Sigma)$ are strictly less than $\cardinality{\Sigma}$ for all $\rho \in \specialRepresentationsSymbol$.
	
	To study this problem, we introduce the following definitions. We use the notation $\innerProduct{f_1}{f_2}_G \define \frac{1}{\cardinality{G}}\sum_{g \in G}f_1(g)\overline{f_2(g)}$, for $f_1, f_2 \colon G \to \complex$.
	\begin{definition}
		Let $G$ be a finite group, and let $\phi$ and $\psi$ be some characters of $G$. We write $\defineNotation{\phiContainsPsi{\phi}{\psi}}$ if there exists some character $\phi'$ of $G$ such that $\phi = \psi + \phi'$.
		
		If $\psi$ is an irreducible character, $\phiContainsPsi{\phi}{\psi}$ if and only if $\psi$ is a summand of $\phi$ in its decomposition to irreducible characters, namely, if and only if $\innerProduct{\phi}{\psi}_G > 0$.
	\end{definition}

	\begin{definition}
		Let $G$ be a finite group, let $\specialCharactersSymbol$ be a set of characters of $G$ and let $K$ be a subgroup of $G$. We write \defineTextNotation{\textit{$\specialCharactersSymbol$ marks $K$}} if there exists some $\chi \in \specialCharactersSymbol$ such that $\chiContainsTrivial{\chi}{K}$.
		
		If $\specialCharactersSymbol = \{\chi\}$, we may omit the curly brackets and write that $\chi$ marks $K$.
	\end{definition}

	The following proposition shows the connection between these definitions and the eigenvalues of $A_\rho(\Sigma)$. Note that the proposition considers only a subset of the nontrivial irreducible representations of $G$, unlike \autoref{proposition:characters_eigenvalues_generation} that considers every nontrivial irreducible representation.
	\begin{proposition} \label{proposition:characters_on_subgroups_and_generation}
		Let $G$ be a finite group, and let $\specialCharactersSymbol$ be a set of nontrivial irreducible characters of $G$. For a character $\chi$, let $\rho_\chi$ be the corresponding representation. Let ${\specialRepresentationsSymbol \define \{\rho_\chi \colon \chi \in \specialCharactersSymbol\}}$, and assume that \textCMarksProperSubgroups{$\specialCharactersSymbol$}{$G$}. Let $\Sigma \subseteq G$ be a symmetric subset. Then the following are equivalent:
		\begin{enumerate}
			\item\label{proposition:characters_on_subgroups_and_generation_generates} The subset $\Sigma$ generates $G$.
			\item\label{proposition:characters_on_subgroups_and_generation_spectral_gap} The graph $\cayleyGraph{G}{\Sigma}$ admits a spectral gap.
			\item\label{proposition:characters_on_subgroups_and_generation_maximal_eigenvalue} The maximal eigenvalue of $A_\rho(\Sigma)$ is strictly less than $\cardinality{\Sigma}$, for all $\rho \in \specialRepresentationsSymbol$.
		\end{enumerate}
	\end{proposition}

	\begin{proof}
		The implications (\ref{proposition:characters_on_subgroups_and_generation_generates}) $\iff$ (\ref{proposition:characters_on_subgroups_and_generation_spectral_gap}) $\implies$ (\ref{proposition:characters_on_subgroups_and_generation_maximal_eigenvalue}) follow from the discussion above. Thus, to finish the proof, we only need to prove the implication (\ref{proposition:characters_on_subgroups_and_generation_maximal_eigenvalue}) $\implies$ (\ref{proposition:characters_on_subgroups_and_generation_generates}). We prove that if $\Sigma$ does not generate $G$, there exists some $\rho \in \specialRepresentationsSymbol$ for which the largest eigenvalue of $A_\rho(\Sigma)$ is equal to $\cardinality{\Sigma}$.
		
		Let $K \define \groupSpan{\Sigma}$, and assume $K \properSubgroup G$. Then, there exists some $\chi \in \specialCharactersSymbol$ for which \textChiDoesntVanishOnK{\chi}{K}. Let $\rho$ be its representation.
		
		Since $\Sigma \subseteq K$, $A_\rho(\Sigma) = A_{\restricted{\rho}{K}}(\Sigma)$, we deduce that the eigenvalues of $A_\rho(\Sigma)$ contain the single eigenvalue of $A_{\trivialCharacter{K}}(\Sigma)$, which is precisely $\cardinality{\Sigma}$.
	\end{proof}
	
	\paragraph{Remark} The condition that $\specialCharactersSymbol$ contains only nontrivial characters is important, since for every group $G$ the character \textCMarksSubgroups{$\trivialCharacter{G}$}{$G$}.
	
	\medskip
	
	\begin{example} \label{example:four_transitive_characters_of_sn}
		Using this proposition, we see that the classification of finite multiply transitive groups (e.g. {\cite[Chapter 7.7]{permutation_groups}} or {\cite[Theorem 5.3]{permutation_groups_2_transitive}}) gives a subset $\overline{\specialCharacterSet{\symmetric{n}}}$ of characters of 
		$\symmetric{n}$, for $n \ge 25$, such that \textCMarksProperSubgroups{$\overline{\specialCharacterSet{\symmetric{n}}}$}{$\symmetric{n}$}.
		
		Let $n \ge 25$, and let $\chi'$ be the character of the permutation representation associated with the action of $\symmetric{n}$ on tuples of $4$ distinct elements of $\{1,\dots,n\}$. Denote $\chi = \chi' - \trivialCharacter{\symmetric{n}}$. It is easy to see (with the details given in \autoref{proposition:stabilizer-induced-character}) that for all $K \properSubgroup G$, \textChiVanishesOnK{\chi}{K} if and only if $K$ acts transitively on these $4$-tuples. This is equivalent to saying that $K$ is $4$-transitive with respect to this action. By the classification, this implies that $K$ is either $\symmetric{n}$ or the alternating group $\alternating{n}$.
		
		Therefore, we define $\overline{\specialCharacterSet{\symmetric{n}}}$ to be the collection of irreducible summands of $\chi$, along with the sign character. Then, \textCMarksProperSubgroups{$\overline{\specialCharacterSet{\symmetric{n}}}$}{$\symmetric{n}$}. By calculating the decomposition of $\chi$ to irreducible characters, we deduce that the characters in $\overline{\specialCharacterSet{\symmetric{n}}}$ are parametrized\declareFootnote{sn_characters_reference}{\cite{sn_characters} is a good reference for this connection between partitions and characters of $\symmetric{n}$, along with the representation theory of $\symmetric{n}$ in general.} by the following\declareFootnote{sn_result}{We present a smaller set of characters with $8$ elements in \autoref{section:sn}.} $12$ partitions of $n$: 
		
		$\begin{Bmatrix}
			(n-1, 1), &(n-2, 2), &(n-2, 1^2), &(n-3, 3), \\
			(n-3, 2, 1), &(n-3, 1^3), &(n-4, 4), &(n-4, 3, 1), \\
			(n-4, 2^2), &(n-4, 2, 1^2), &(n-4, 1^4), &(1^{n})
		\end{Bmatrix}.$
	\end{example}
	
	\paragraph{ }
	For a prime power $q$, let $\field{q}$ be the field with $q$ elements. In this paper, we prove similar results about the general linear group over a finite field $\GL{n}{q}$ and its relatives -- $\PGL{n}{q}$, $\SL{n}{q}$ and $\PSL{n}{q}$. For $n \ge 3$, the group $\GL{n}{q}$ acts $2$-transitively on the projective space $\projective^{n-1}\field{q}$, and every subgroup of $\GL{n}{q}$ with a $2$-transitive action on the projective space contains $\SL{n}{q}$ (\cite{2_transitive_collineation_groups}), except in the exceptional case where $n = 4$ and $q = 2$. Thus, we may attempt to apply the methods we used for $\symmetric{n}$ on these groups, along with the classification of multiply transitive groups. However, two differences with the case of $\symmetric{n}$ arise.
	
	First, in the case of $\symmetric{n}$, we define a set of characters $\specialCharacterSet{\symmetric{n}}$ \textCMarksSubgroupsThatDoNotContainH{that}{$\symmetric{n}$}{$\alternating{n}$}, and then we add the sign character, which satisfies \textChiDoesntVanishOnK{\sgn}{\alternating{n}}, to obtain $\overline{\specialCharacterSet{\symmetric{n}}}$.
	
	In the case of $\GL{n}{q}$, the same arguments would give a set of characters $\specialCharacterSetGLnq{n}{q}$ \textCMarksSubgroupsThatDoNotContainH{that}{$\GL{n}{q}$}{$\SL{n}{q}$}. In general, there exists no single character we can add to $\specialCharacterSetGLnq{n}{q}$ to obtain some set $\overline{\specialCharacterSetGLnq{n}{q}}$ \textCMarksAllProperSubgroups{that}{$\GL{n}{q}$} (see \autoref{corollary:XGLn_lower_bound}), and the number of characters we need to add can be very large for values of $q$ where $q-1$ has a large number of prime factors.
	
	Thus, we only find a set of characters $\specialCharacterSetGLnq{n}{q}$ \textCMarksSubgroupsThatDoNotContainH{that}{$\GL{n}{q}$}{$\SL{n}{q}$} in this paper. Then, by restricting the characters in $\specialCharacterSetGLnq{n}{q}$ to $\SL{n}{q}$ (all of the characters in this set reduce irreducibly), we obtain a set of characters $\specialCharacterSetSLnq{n}{q}$ of $\SL{n}{q}$ \textCMarksProperSubgroups{that}{$\SL{n}{q}$}.
	
	Second, the character of the permutation representation associated with the action of $\GL{n}{q}$ on pairs of projective vectors in $\projective^{n-1}\field{q}$ decomposes to ${q+3}$ distinct irreducible characters for every $n \ge 4$. Our aim is to obtain sets $\specialCharacterSetGLnq{n}{q}$ of nontrivial irreducible characters whose sizes are bounded uniformly in $n$ and $q$.
	
	Therefore, we take a slightly different approach. Then, we obtain sets $\specialCharacterSetGLnq{n}{q}$ with $\cardinality{\specialCharacterSetGLnq{n}{q}} \le 5$. 
		
	\subsection{Main Results} \label{section:main_results}
	\paragraph{ } In the representation theory of $\GL{n}{q}$, several classes of characters arise. One of these classes is the class of unipotent characters, which are irreducible characters parametrized by partitions of $n$ (for more information about the representation theory of $\GL{n}{q}$, see \autoref{section:glnq_characters}).
	
	The unipotent characters are closely related to the action of $\GL{n}{q}$ on the projective space $\projective^{n-1}\field{q}$, and have a fundamental role in the group's representation theory. For example, $\unipotentCharacter{(n-1,1)}(g)$ counts the number of fixed points of $g$ in this action, minus $1$. The character $\unipotentCharacter{(n-2, 2)}(g)$ counts the number of $2$-dimensional projective subspaces $W$ with $gW=W$, minus the number of fixed points of $g$.
	
	It turns out that the unipotent characters can be defined in a natural way for the projective and special groups as well -- see \autoref{section:special_linear_unipotent_restriction} and \autoref{section:projective_unipotent_restriction}.	Thus, by abuse of notation, we define
	\[\defineNotation{\specialCharacterSetXLnq{G}} \define \{\unipotentCharacter{(n-3, 1^3)}, \unipotentCharacter{(n-3, 2, 1)}, \unipotentCharacter{(n-2, 1^2)}, \unipotentCharacter{(n-2, 2)}, \unipotentCharacter{(n-1, 1)}\} \subseteq \irreducibleCharacters{G},\]
	for $G$ either $\GL{n}{q}$ or one of its relatives, when $n \ge 3$ and $(n, q) \ne (4, 2)$. The character $\unipotentCharacter{\alpha}$ is omitted if $\alpha$ is not a partition\declareFootnote{alpha_omitted}{For example, when $n=3$, $(n-2, 2) = (1, 2)$ which is not a partition, because the numbers are not non-increasing.}.
	
	\begin{theorem} \label{theorem:nonvanishing_characters_of_gln}
		Let $n \ge 3$ and $q$ a prime power, such that $(n, q) \ne (4, 2)$. Let $G$ be either $\GL{n}{q}$ or $\PGL{n}{q}$, and let $G' \subgroup G$ be $\SL{n}{q}$ or $\PSL{n}{q}$ respectively. Then, \textCMarksSubgroupsThatDoNotContainH{$\specialCharacterSetXLnq{G}$}{$G$}{$G'$}.
	\end{theorem}

	\begin{theorem} \label{theorem:nonvanishing_characters_of_sln}
		Let $n \ge 3$ and $q$ a prime power, such that $(n, q) \ne (4, 2)$. Let $G$ be either $\SL{n}{q}$ or $\PSL{n}{q}$. Then, \textCMarksProperSubgroups{$\specialCharacterSetXLnq{G}$}{$G$}.
	\end{theorem}

	Note that we do not claim that these sets $\specialCharacterSet{G}$ are unique, nor that they are minimal.

	\paragraph{$\mathbf{\GL{4}{2}}$} For completeness, we present the following theorem that deals with the exceptional case of the previous theorems. For this case, let $\psi$ be the single irreducible character of degree $7$ of $\GL{4}{2} \cong \alternating{8}$ (the character of the permutation representation associated with the action on $8$ elements, minus the trivial). Here, $q = 2$ implies that $\GL{4}{2}$ is equal to all of its relatives. Denote
	
	\[\defineNotation{\specialCharacterSetGLnq{4}{2}} \define \{\unipotentCharacter{(1^4)}, \unipotentCharacter{(2, 1^2)}, \unipotentCharacter{(2^2)}, \unipotentCharacter{(3, 1)}, \psi\} \subseteq \irreducibleCharacters{\GL{4}{2}}.\]

	\begin{theorem} \label{theorem:nonvanishing_characters_of_gl42}
		\textCMarksProperSubgroups{$\specialCharacterSetGLnq{4}{2}$}{$\GL{4}{2}$}.
	\end{theorem}

	\paragraph{$\mathbf{\GL{2}{q}}$} We present here the case of $n = 2$ as well. In this case, there exists a single nontrivial unipotent character, $\unipotentCharacter{(1, 1)}$, which does not encode enough information about subgroups of $\GL{2}{q}$, since \textChiVanishesOnK{\unipotentCharacter{(1, 1)}}{K} for every $K \subgroup \GL{2}{q}$ that acts transitively on $\projective^{1}\field{q}$. There are such subgroups that do not contain $\SL{2}{q}$. For example, $\GL{1}{q^2}$ acts transitively on $\field{q^2} \setminus 0$. By the isomorphism $\field{q^2} \cong \field{q}^2$ of vector spaces over $\field{q}$. We get an inclusion ${\GL{1}{q^2} \properSubgroup \GL{2}{q}}$. Thus, $\GL{2}{q}$ has transitive subgroups that do not contain $\SL{2}{q}$, and an additional character is required.
	
	For $G$ either $\GL{2}{q}$ or one of its relatives, consider the irreducible character $\additionalForGLTwo$, defined in \autoref{definition:additional_character_2}. Denote
	\[\specialCharacterSetXLnq{G} \define \{\unipotentCharacter{(1, 1)}, \additionalForGLTwo \} \subseteq \irreducibleCharacters{G}.\]
	
	\begin{theorem} \label{theorem:nonvanishing_characters_of_gl2}
		Let $q$ be a prime power, let $G$ be either $\GL{2}{q}$ or $\PGL{2}{q}$, and let $G' \subgroup G$ be $\SL{2}{q}$ or $\PSL{2}{q}$ respectively. Then, \textCMarksSubgroupsThatDoNotContainH{$\specialCharacterSetXLnq{G}$}{$G$}{$G'$}.
	\end{theorem}

	\begin{theorem} \label{theorem:nonvanishing_characters_of_sl2}
		Let $q$ be a prime power, and let $G$ be either $\SL{2}{q}$ or $\PSL{2}{q}$. Then, \textCMarksProperSubgroups{$\specialCharacterSetXLnq{G}$}{$G$}.
	\end{theorem}

	\subsection{Proof Strategy}
	\paragraph{ } Consider the family of subgroups $\GL{n}{q}$ with $n \ge 3$, along with their special subgroups $\SL{n}{q} \subgroup \GL{n}{q}$. Denote by $\irreducibleCharacters{\GL{n}{q}}$ the collection of irreducible characters of $\GL{n}{q}$. We wish to find sets of nontrivial characters ${\specialCharacterSetGLnq{n}{q} \subseteq \irreducibleCharacters{\GL{n}{q}}}$ of bounded size (with respect to both $n$ and $q$), \textCMarkSubgroupsThatDoNotContainH{that}{$\GL{n}{q}$}{$\SL{n}{q}$}.
	
	The following notion of transitivity plays an important role in the proof:
	\begin{definition} \label{definition:g_transitivity}
		Let $G$ be a finite group acting on some set $X$, let ${K \subgroup G}$, let ${\ell \in \positiveIntegers}$, and let $\xOneToXk{x}{\ell}$ be distinct elements of $X$. We say that $K$ is \defineTextNotation{\xkGTransitive{x}{\ell}{$G$}} if for every $g \in G$, there exists $k \in K$ such that ${\forall i \colon kx_i = gx_i}$.
	\end{definition}
	\begin{example}
		\begin{enumerate}
			\item \xkGTransitivity{x}{\ell}{$G$} is equivalent to transitivity on the orbit $\orbit{G}{(x_1,\dots, x_\ell)}$ in the diagonal action of $G$ on $X^\ell$.
			\item For $\ell \le n$, a subgroup $K \le \symmetric{n}$ is \xkGTransitive{}{\ell}{$\symmetric{n}$} if and only if the action of $K$ on $\{1,\dots,n\}$ is $\ell$-transitive.
			\item If the action of $G$ on $X$ is $\ell$-transitive, $K \subgroup G$ is \xkGTransitive{x}{\ell}{$G$} if and only if the action of $K$ is  $\ell$-transitive.
			\item Let $\projectiveVector{e}_i$ be the point in $\projective^{n-1}\field{q}$ represented by the $i$th element of the standard basis of $\field{q}^n$. For $\ell \le n$, $K \subgroup \GL{n}{q}$ is
			\xkGTransitive{\projectiveVector{e}}{\ell}{$\GL{n}{q}$} if and only if for every pair of $\ell$-tuples of linearly independent projective vectors -- $(x_1, \dots, x_\ell)$ and $(y_1, \dots, y_\ell)$ -- there exists an element $k \in K$ such that ${\forall i \colon kx_i = y_i}$.
			\item $\SL{n}{q}$ is \xkGTransitive{\projectiveVector{e}}{n}{$\GL{n}{q}$}.
		\end{enumerate}
	\end{example}

	We can now describe the framework for the proof. Let $G_n= \GL{n}{q}$ for $n \ge 3$, acting on the sets $X_n = \projective^{n-1}\field{q}$. The sets $X_n$ are increasing with respect to inclusion $i_n\colon \projective^{n-1}\field{q} \hookrightarrow \projective^{n}\field{q}$, ${i_n(x_1,\dots,x_n) = (x_1,\dots,x_n, 0)}$. Denote ${Y \define X_3 = \projective^2\field{q}}$, and consider the stabilizer $\stab{(G_n)}{Y} = {\{g \in G_n \colon gY = Y\}}$, which can be realized as the group of upper diagonal block matrices 
	\[
	P_{(3, n-3)} \define
	\begin{bmatrix}
	\GL{3}{q}	& *					\\
	0			& \GL{n - 3}{q}  \\
	\end{bmatrix} 
	\subgroup \GL{n}{q}.
	\]
	It projects to the symmetry group of $Y$, and we denote its image by $H$. Note that $H \cong \PGL{3}{q}$.
	
	Given a character of $H$, we pull it back using the projection from the stabilizer. Here we take a character of a small group, with a relatively small degree, and lift it to a much larger group without increasing its degree, nor the number of its irreducible summands. We get a character of $\stab{(G_n)}{Y}$ which is then induced to the group $G_n$.
	
	This allows us to study $H$ and its characters, and then use the described process to lift them to the sequence of groups $G_n$ and obtain results on these characters. Thus, assume that we have some subset $\specialCharacterSet{H} \subseteq \irreducibleCharacters{H} = \irreducibleCharacters{\PGL{3}{q}}$, such that \textCMarksSubgroups{$\specialCharacterSet{H}$}{$H$} that is not \xkGTransitive{\projectiveVector{e}}{3}{$H$}. Lift the characters to characters of $\stab{(G_n)}{Y}$, then induce them to characters of $G_n$. Denote the induced characters by $\specialCharacterSet{G_n}'$.
	
	We prove that since $H$, $\specialCharacterSet{H}$, $G_n$ and $Y$ satisfy several properties, which are described below (with $\ell = 3$), \textCMarksSubgroups{$\specialCharacterSet{G_n}'$}{$G_n$} that is not \xkGTransitive{\projectiveVector{e}}{2}{$G_n$}.
	
	It turns out that this implies that \textCMarksSubgroupsThatDoNotContainH{$\specialCharacterSet{\GL{n}{q}}'$}{$\GL{n}{q}$}{$\SL{n}{q}$}, outside the exceptional case of ${q = 2}, {n = 4}$. Therefore, to finish the proof of the main theorems, it remains to decompose the characters in $\specialCharacterSet{G_n}'$ to irreducible characters, and deduce the theorems.
	
	The method described above can be used for every sequence of groups $G_n$ acting on sets $X_n$, which are increasing with respect to inclusion. In the general case, we need the following three conditions to hold in order to apply this method:
	
	\begin{enumerate}
		\item Y contains some $\ell$ element subset $Y_0 = \{a_1, \dots, a_\ell\}$, which can be thought of as the basis of $Y$.
		\item The translations of $Y$ by elements of $G_n$ cover all $\ell$-element subsets of $X_n$.
		\item If we can translate less than $\ell$ of the elements of $Y_0$, with some element of $G_n$, while keeping them in $Y$, we can translate these same elements in the same way with some element of $\stab{(G_n)}{Y}$.
	\end{enumerate}

	Then, we do the following to finish the proof:
	\begin{enumerate}
		\item Find some subset of characters $\specialCharacterSet{H} \subseteq \irreducibleCharacters{H}$ \textCMarksSubgroupsThatAreNotxkTransitive{that}{$H$}{a}{\ell}.
		\item Decompose the characters of $\specialCharacterSet{G_n}'$ to irreducibles, to obtain some set $\specialCharacterSet{G_n}$.
		\item Classify\declareFootnote{classification_not_done}{The classification in the case of $\GL{n}{q}$ is actually implied by the classification of multiply transitive groups, and is not done in this paper} all \xkGTransitive{a}{\ell-1}{$G_n$} subgroups of $G_n$.
	\end{enumerate}
	
	To finish the proof of the main theorems, we need to make sure the decomposed characters $\specialCharacterSet{\GL{n}{q}}$ restrict to irreducible characters of $\SL{n}{q}$ and project to irreducible characters of $\PGL{n}{q}$ and $\PSL{n}{q}$ for almost all pairs $n \in \positiveIntegers$, $q$ a prime power. Then we deal with the exceptional cases that arise.
	
	\subsection{Paper Organization}
	\paragraph{ } In \autoref{section:induced_characters_of_stab}, we study characters induced by action stabilizers, which are the characters of permutation representations. The main result of this section is \autoref{corollary:glnq-stabilizer-character}. To prove it, we prove several technical lemmas. In general, we prove that there exist certain characters such that if their restriction to some $K \le G$ does not contain the trivial representation, the action of $K$ has certain transitivity properties.
	
	In \autoref{section:glnq_characters}, we introduce general facts about the characters of $\GL{n}{q}$, and study certain characters that are important to the rest of the paper.
	
	In \autoref{section:maximal_subgroups_of_glnq}, we study maximal subgroups of $\GL{n}{q}$. \autoref{section:geometric_subgroups_of_glnq_description} describes the families of a class of maximal subgroups called the geometric subgroups, \autoref{section:geometric_subgroups_of_glnq_restriction} restricts certain characters to the geometric subgroups and proves these do not contain the trivial representation, and \autoref{section:non_geometric_subgroups} handles the remaining case of the non-geometric maximal subgroups.
	
	In \autoref{section:nonvanishing_characters}, we prove the main theorems for $n \ge 3$: for the general case (\autoref{section:nonvanishing_characters_ge_3}), the special cases (\autoref{section:special_linear_unipotent_restriction}), the projective cases  (\autoref{section:projective_unipotent_restriction}), and the exceptional case of $\GL{4}{2}$ (\autoref{section:gl42}).
	
	In \autoref{section:nonvanishing_characters_2}, we prove the main theorems for $n = 2$.
	
	In \autoref{section:open_questions}, we present several open questions raised in this paper.
	
	In \autoref{section:sn}, we apply the method described in the introduction on $\symmetric{n}$ and state the results.
	
	In \autoref{section:group_extensions}, we consider the situation where $G$ is some finite group, with $N \normalSubgroup G$. We show how sets of nontrivial irreducible characters $\specialCharacterSet{G}$, $\specialCharacterSet{\factor{G}{N}}$ and $\specialCharacterSet{N}$ behave with respect to one another, where these sets mark proper subgroups of $G$, $\factor{G}{N}$ and $N$ respectively.
	
	\subsection{Notations}
	\begin{tabular}{Cc}
		\text{Symbol} & Definition\\ \hline
		\trace{\mathbb{E}}{\mathbb{F}} & The trace of the field extension $\mathbb{E}/\mathbb{F}$ \\
		\projective V & The projective space of $V$ \\
		\projectiveVector{v} & The point in the projective space represented by $v$ \\
		\characters{G} & The collection of characters of $G$ (non-virtual)\\
		\irreducibleCharacters{G} & The collection of irreducible characters of $G$ \\
		\innerProduct{u}{v} & The inner product of $u$ and $v$ \\
		\innerProduct{\psi}{\phi}_G & The inner product of $\psi$ and $\phi$ in the character ring of $G$ \\
		\trivialCharacter{G} & The trivial character of $G$ \\
		\linearSpan{v_1, \dots, v_k} & The linear span of $v_1, \dots, v_k$ \\
		e_i \in \field{q}^n & The $i$th element of the standard basis of $\field{q}^n$\\
		\groupSpan{g_1, \dots, g_k} & The subgroup generated by $g_1, \dots, g_k \in G$ \\
		\induced{\chi}{K}{G} & The induced character from $K$ to $G$ \\
		\groupCenter{G} & The center of $G$ \\
		\symmetric{n} & The symmetric group on $n$ elements \\
		\alternating{n} & The alternating group on $n$ elements \\
		G \acts X & $G$ acts on $X$ \\
		\stab{G}{Y} & The stabilizer of $Y$ -- $\{g \in G \colon gY = Y\}$, where $G \acts X$ and $Y \subseteq X$ \\
		\orbit{G}{Y} & The orbit of $Y$ -- $\{gy \colon g \in G, y \in Y\}$ where $G \acts X$ and $Y \subseteq X$ \\
		\normalizer{K}{G} & The normalizer of $K \subgroup G$ \\
		\lambda \vdash n & $\lambda$ is a partition of $n$ \\
		\lambda' & The partition conjugate to $\lambda \vdash n$ \\
	\end{tabular}

	\section{Induced Characters of Action Stabilizers} \label{section:induced_characters_of_stab}
	\paragraph{ } In this section, we relate the notion of \xkGTransitivity{y}{\ell}{$G$} and the properties of characters of groups. We start by stating several standard results in representation theory that we need.
	\begin{theorem}[Frobenius Reciprocity] \label{frobenius_reciprocity}
		Let $G$ be a group and let $K \subgroup G$. Let $\psi \in \characters{K}$ and $\phi \in \characters{G}$. Then
		\[\innerProduct{\induced{\psi}{K}{G}}{\phi}_G = \innerProduct{\psi}{\restricted{\phi}{K}}_K.\]
	\end{theorem}

	\begin{proposition}[e.g. {\cite[Equation 3.18]{representation_theory_fulton_harris}}] \label{proposition:induced_character_formula}
		Let $G$ be a finite group, let $K \subgroup G$, and let $\chi \in \characters{K}$. Then
		\[\induced{\chi}{K}{G}(g) = \frac{1}{\cardinality{K}} \sum_{t \in G \mathSuchThat \\ t^{-1}gt \in K}\chi(t^{-1}gt).\]
	\end{proposition}

	\begin{lemma}[Burnside's Lemma] \label{burnsides_lemma}
		Let $G$ be a finite group acting on some set $X$. For each $g \in G$, denote by $\fixXg{X}{g} \define \{x \in X \colon gx = x\}$. Let $\cardinality{\rfactor{X}{G}}$ denote the number of orbits of the action. Then
		\[\cardinality{\rfactor{X}{G}} = \frac{1}{\cardinality{G}}\sum_{g \in G}\cardinality{\fixXg{X}{g}}.\]
	\end{lemma}
	
	\begin{proposition} \label{proposition:stabilizer-induced-character}
		Let $G$ be a finite group acting transitively on some set $X$, let $K \subgroup G$, and let $\chi$ be the character of the permutation representation associated with the action, namely, $\chi(g)$ is the number of $x \in X$ fixed by $g$. Then, ${\characterSum{\chi}{K} = 1}$ $\iff$ $K$ acts transitively on $X$.
	\end{proposition}
	\begin{proof}
		By \nameref{burnsides_lemma}, 
		\begin{equation*}
			\characterSum{\chi}{K} = \frac{1}{\cardinality{K}}\sum_{g \in K}\chi(g) = \frac{1}{\cardinality{K}}\sum_{g \in K}\cardinality{\fixXg{X}{g}} = \cardinality{\rfactor{X}{K}}.
		\end{equation*}
		The right hand side is $1$ if and only if $K$ acts transitively on $X$.
	\end{proof}

	The equality $\characterSum{\chi}{K} = 1$ means that $\chiKContainsTrivial{\chi}{K}$ and $\trivialCharacter{K}$ has multiplicity $1$ in $\restricted{\chi}{K}$.

	\begin{remark} \label{remark:permutation-representation}
		The character of the permutation representation associated with a transitive group action is equal to the character $\inducedTrivial{\stab{G}{x}}{G}$, induced from the stabilizer $\stab{G}{x}$ for every $x \in X$.
	\end{remark}
	
	\begin{corollary} \label{corollary:permutation-character-marks-transitive}
		Let $G$ be a finite group acting transitively on some set $X$, and let $\chi$ be the character of the permutation representation associated with the action. Denote $\chi' \define \chi - \trivialCharacter{G}$. Then, \textCMarksSubgroups{$\chi'$}{$G$} that does not act transitively on $X$.
	\end{corollary}

	\paragraph{Notation for the central example} Here, we introduce notation useful for the remainder of this section. We start by an example and then define the notation in general.
	
	Let $G = \GL{n}{q}$ acting on $X = \projective^{n-1}\field{q}$, and let $\ell \le n$. Consider the projective subspace $Y = \projective\left(\linearSpan{e_1, \dots, e_\ell}\right) \subseteq X$, whose translations by elements of $G$ are precisely the $\ell$-dimensional projective subspaces of $X$.	The stabilizer $\stab{G}{Y}$ is
	\[
	P_{(\ell, n-\ell)} \define
	\begin{bmatrix}
	\GL{\ell}{q}	& *					\\
	0				& \GL{n-\ell}{q}  \\
	\end{bmatrix} 
	\subgroup \GL{n}{q}.
	\]
	
	Let $p \colon P_{(\ell, n-\ell)} \to \PGL{\ell}{q}$ be the projection to the upper left block, composed with the projection to the projective group. This corresponds to the restriction of the action of $\stab{G}{Y}$ to $Y$, with $p$ being the projection from $\stab{G}{Y}$ to $\symmetries{Y}$, the symmetric group on the set $Y$. We denote by $H$ the image of this morphism -- $\PGL{\ell}{q}$.
	
	For every $t \in \GL{n}{q}$, we define $p_t$ and $H_t$ in a similar way -- by projecting from $\stab{G}{tY}$ to $\symmetries{tY}$. Note that the translations $tY$ of $Y$ are the $\ell$-dimensional projective subspaces of $X$.
	
	For every $t \in \GL{n}{q}$, the group $H_t$ is isomorphic to $\PGL{\ell}{q}$, as this is the group of linear morphisms of the $\ell$-dimensional projective space $tY$. Let $f_t \colon H \to H_t$ be the natural isomorphisms, defined by lifting an element of $H$ to some element of $\stab{G}{Y}$, conjugating by $t$ and projecting to $H_t$. The resulting element of $H_t$ is independent of the choice of the element of $\stab{G}{Y}$.
	
	Let $\chi$ be a character of $H = \PGL{\ell}{q}$. The pull-back of $\chi$ to $\stab{G}{Y}$ through the projection then defines a new character $\hcLikePullback{\chi}{H}{\stab{G}{Y}} \in \characters{\stab{G}{Y}}$, which is applied on an element of $\stab{G}{Y}$ as follows:
	
	\[
	\hcLikePullback{\chi}{H}{\stab{G}{Y}} \left(
	\begin{bmatrix}
	A	& B	\\
	0	& C \\
	\end{bmatrix}
	\right) = \chi(A)
	.
	\]
	The character $\hcLikePullback{\chi}{H}{\stab{G}{Y}}$ is then induced to $G$ to define a character ${\hcLikeInduction{\chi}{H}{G} \in \characters{G}}$.
	
	\paragraph{ }	
	Before we move to the general case, we define the pull-back of a character. Let $G$ and $H$ be finite groups. For a character $\chi$ of $H$ and a morphism ${f \colon G \to H}$, we denote by $\defineNotation{\pullbackChiByF{\chi}{f}}$ the pull-back of $\chi$ by $f$. The pull-back is a character of $G$, with $\pullbackChiByF{\chi}{f}(g) = \chi(f(g))$.
	
	\begin{notation} \label{notation:section2}
		Given a finite group $G$ acting on some set $X$, a subset $Y \subseteq X$ and some $t \in G$, define the following.
		\begin{enumerate}[label*=\arabic*.]
			\item We denote by $\defineNotation{H_t} \subgroup \symmetries{tY}$ the restriction of the action of $\stab{G}{tY}$ on $tY$ (formally -- the quotient of $\stab{G}{tY}$ by the kernel of its action on $tY$).
			\item The corresponding epimorphism is denoted by $\defineNotation{p_t} \colon \stab{G}{tY} \twoheadrightarrow H_t$.
			\item We specialize the case when $t = \unit$, the identity of $G$, and denote $\defineNotationsSilent{H, p}$ ${H \define H_\unit}$ and ${p \define p_\unit}$.
			\item Denote by $\defineNotation{f_t} \colon H \to H_t$ -- the "conjugation maps"\declareFootnote{f_t_well_defined}{The morphisms $f_t$ are well defined, as $a \in \ker(p) \Rightarrow tat^{-1} \in \ker(p_t)$.} $f_t(g) \define p_t(tp^{-1}(g)t^{-1})$, which make the following diagram commute. Note that $f_t$ is an isomorphism.
			\begin{center}
			\begin{tikzpicture}[scale=1.5]
				\node (GY) at (0,1) {$\stab{G}{Y}$};
				\node (GtY) at (1,1) {$\stab{G}{tY}$};
				\node (H) at (0,0) {$H$};
				\node (Ht) at (1,0) {$H_t$};
	
				\node (H-Y) at (-0.25,0) {\reflectbox{$\acts$}};
				\node (Y) at (-0.5, 0) {$Y$};
	
				\node (Ht-tY) at (1.25,0) {$\acts$};
				\node (tY) at (1.5, 0) {$tY$};
				
				\path[->,font=\scriptsize]
				(GY) edge node[above]{$g \mapsto tgt^{-1}$} (GtY)
				(H) edge node[below]{$f_t$} (Ht);
				
				\path[->>,font=\scriptsize]
				(GY) edge node[left]{$p$} (H)
				(GtY) edge node[right]{$p_t$} (Ht);
			\end{tikzpicture}
			\end{center}
			
			\item For a character $\chi$ of $H$:
			\begin{enumerate}[label*=\arabic*.]
			\item Define $\chi_t \define \pullbackChiByF{\chi}{(f_t^{-1})} \in \characters{H_t}$, that is, ${\chi_t(g) = \chi(f^{-1}_t(g))}$.
			\item Define $\hcLikePullback{\chi}{H}{\stab{G}{Y}} \define \pullbackChiByF{\chi}{p}$, that is, $\hcLikePullback{\chi}{H}{\stab{G}{Y}}(g) \define \chi(p(g))$.
			\item Define\declareFootnote{harish_chandra_remark}{This definition is similar to the definition of $\circ$, described in \autoref{section:glnq_characters}.} $\hcLikeInduction{\chi}{H}{G} \in \characters{G}$ as the induction of $\hcLikePullback{\chi}{H}{\stab{G}{Y}}$ to $G$, namely, ${\hcLikeInduction{\chi}{H}{G} \define \induced{\hcLikePullback{\chi}{H}{\stab{G}{Y}}}{\stab{G}{Y}}{G}}$.
			\end{enumerate}
		\end{enumerate}
	\end{notation}
	
	\paragraph{ }
	In the rest of this section, we study the induced characters $\hcLikeInduction{\chi}{H}{G}$. Our end goal is the following lemma:
	
	\begin{lemma} \label{lemma:stabilizer-character-permutaitons-global-transitivity}
		Let $G$ be a finite group acting on some set $X$ and let $Y \subseteq X$. Using \autoref{notation:section2} let $\chi \in \characters{H}$. Assume: 
		\begin{enumerate}
			\item $Y\supseteq Y_0=\{a_1 \dots a_\ell\}$
			\item \textCMarksSubgroupsThatAreNotxkTransitive{$\chi$}{$H$}{a}{\ell}.
			\item \label{lemma:stabilizer-character-permutaitons-global-transitivity-covering} The translations of $Y$ cover all $\ell$-element subsets of $X$, that is, for every $B \subseteq X$ with $\cardinality{B} \le \ell$, there exists some $t \in G$ such that $B \subseteq tY$.
			\item \label{lemma:stabilizer-character-permutaitons-global-transitivity-permutation} For every $t \in G$ and $i < \ell$, if $ta_j \in Y$ for all $j \le i$, then there exists some $s \in \stab{G}{Y}$ such that $\forall j \le i \colon sa_j = ta_j$.
		\end{enumerate}
		
		Then \textCMarksSubgroupsThatAreNotxkTransitive{$\hcLikeInduction{\chi}{H}{G}$}{$G$}{a}{\ell-1}.
	\end{lemma}

	We apply the lemma on $G = \GL{n}{q}$ acting on $X = \projective^{n-1}\field{q}$, along with ${Y = \projective\left(\linearSpan{e_1, \dots, e_\ell}\right) \subseteq X}$ and $Y_0 = \{\projectiveVector{e}_1,\dots,\projectiveVector{e}_\ell\}$. Since $\PSL{\ell}{q}$ is \xkGTransitive{\projectiveVector{e}}{\ell}{$\PGL{\ell}{q}$}, if $\chi$ is a character of $\PGL{\ell}{q}$ \textCMarksSubgroupsThatDoNotContainH{that}{$\PGL{\ell}{q}$}{$\PSL{\ell}{q}$}, \textCMarksSubgroupsThatAreNotxkTransitive{$\chi$ also}{$\PGL{\ell}{q}$}{\projectiveVector{e}}{\ell}.
	
	\begin{corollary} \label{corollary:glnq-stabilizer-character}
		Let $\ell \le n$ be integers, let $G = \GL{n}{q}$, $H = \PGL{\ell}{q}$, and let $K \subgroup G$. Let $\chi \in \characters{H}$ be a character \textCMarksSubgroupsThatDoNotContainH{that}{$H$}{$\PSL{\ell}{q}$}. Then, \textCMarksSubgroupsThatAreNotxkTransitive{$\hcLikeInduction{\chi}{H}{G}$}{$G$}{\projectiveVector{e}}{\ell-1}.
	\end{corollary}
	
	\paragraph{ } \autoref{lemma:stabilizer-character-permutaitons-global-transitivity} is not limited to the case of $\GL{n}{q}$. For example, by applying the lemma for $G = S_n, Y = Y_0 = \{1,\dots,\ell\}$, we deduce:
	
	\begin{corollary} \label{corollary:sn-stabilizer-character}
		Let $\ell \le n$ be integers, and let $K \subgroup \symmetric{n}$. Assume there exists some $\chi \in \characters{\symmetric{\ell}}$ such that \textCMarksProperSubgroups{$\chi$}{$\symmetric{\ell}}$. Then,
		\textCMarksSubgroupsThatAreNotkTransitive{$\hcLikeInduction{\chi}{\symmetric{\ell}}{\symmetric{n}}$}{$\symmetric{n}$}{$(\ell-1)$}.
	\end{corollary}

	\paragraph{Remark} \autoref{corollary:sn-stabilizer-character} gives an example where we "lose a transitivity degree", that is, the resulting character $\hcLikeInduction{\chi}{H}{G}$ does not mark all subgroups of $G$ that are not \xGTransitive{$(a_1, \dots, a_{\ell - 1}, a_\ell)$}{$G$}, but only those that are not \xkGTransitive{a}{\ell-1}{$G$}. For example, choose $G = \symmetric{4}$ with the usual action on $\{1, 2, 3, 4\}$, $Y = \{1, 2\}$, $\ell = 2$, $H = \symmetric{2}$ and $\chi$ the single nontrivial irreducible character of $H$. Then, $\hcLikeInduction{\chi}{\symmetric{2}}{\symmetric{4}}$ decomposes as the sum of the characters parametrized by $(3, 1)$ and $(2, 1, 1)$. These are not enough to mark all subgroups of $\symmetric{4}$ that are not $2$-transitive, as neither marks the subgroups of $\symmetric{4}$ of index $3$, which are conjugate to $K=\groupSpan{(1, 2, 3, 4), (1, 3)}$ (with $(x_1, \dots, x_a)$ denoting the cyclic permutation of $x_1, \dots, x_a$).
	
	\medskip

	\paragraph{ } We prove \autoref{lemma:stabilizer-character-permutaitons-global-transitivity} in several steps. Let $K \le G$ and $t \in G$. We consider the action of $K \cap \stab{G}{tY}$ on $tY$. To understand this action, we can consider the projection to $H_t$ -- the group ${J_t \define p_t(K \cap \stab{G}{tY})}$, which is the quotient of $K \cap \stab{G}{tY}$ by the kernel of this action.
	
	In \autoref{lemma:stabilizer-character}, we deduce that if \textChiVanishesOnK{\hcLikeInduction{\chi}{H}{G}}{K}, then \textChiVanishesOnK{\chi_t}{J_t} for all $t \in G$. This allows us to reduce the question about the subgroups and characters of $G$ to a question about the subgroups and characters of $H$.
	
	In \autoref{lemma:stabilizer-character-permutaitons-local-transitivity}, we assume there exist some $a_1, \dots, a_\ell \in Y$ and some character $\chi$ of $H$ \textCMarksSubgroupsThatAreNotxkTransitive{that}{$H$}{a}{\ell}. Then, we show that this $\chi$ can be "pulled-back" to show that \textChiVanishesOnK{\hcLikeInduction{\chi}{H}{G}}{K} implies that $K \cap \stab{G}{tY}$ is \xkGTransitive{a}{\ell}{$\stab{G}{tY}$} for all $t \in G$. We proceed to proving the lemmas.
	
	\begin{lemma} \label{lemma:stabilizer-character}
		Continuing with \autoref{notation:section2}, let $G$ be a finite group acting on some set $X$, let $Y \subseteq X$ be some subset, and let $\chi \in \characters{H}$. For every $K \subgroup G$, we have $\chiDoesntContainTrivial{\hcLikeInduction{\chi}{H}{G}}{K}$ if and only if for every $t \in G$, the group $J_t \define p_t(K \cap \stab{G}{tY})$ satisfies $\chiKDoesntContainTrivial{\chi_t}{J_t}$.
	\end{lemma}
	\begin{proof}
		Let $K \subgroup G$. For every $t \in G$, $J_t = f_t(p((t^{-1}Kt) \cap \stab{G}{Y}))$. Now,
		\begin{multline*}
			\begin{aligned}
				\characterSum{\hcLikeInduction{\chi}{H}{G}}{K} &= \frac{1}{\cardinality{K}}\sum_{g \in K}\hcLikeInduction{\chi}{H}{G}(g)  \stackrel{\mathclap{\normalfont\mbox{(1)}}}{=} \frac{1}{\cardinality{K}}\sum_{g \in K} \frac{1}{\cardinality{\stab{G}{Y}}} \sum_{t \in G \mathSuchThat t^{-1}gt\in \stab{G}{Y}}\hcLikePullback{\chi}{H}{\stab{G}{Y}}(t^{-1}gt) \\
										 &= \frac{1}{\cardinality{K}\cdot \cardinality{\stab{G}{Y}}}\sum_{t \in G} \left( \sum_{g \in K \mathSuchThat t^{-1}gt\in \stab{G}{Y}}\hcLikePullback{\chi}{H}{\stab{G}{Y}}(t^{-1}gt)\right) \\
										 &= \frac{1}{\cardinality{K}\cdot \cardinality{\stab{G}{Y}}}\sum_{t \in G} \left( \sum_{g'\in \stab{G}{Y} \mathSuchThat tg't^{-1} \in K}\hcLikePullback{\chi}{H}{\stab{G}{Y}}(g')\right) \\
 										 &= \frac{1}{\cardinality{K}\cdot \cardinality{\stab{G}{Y}}}\sum_{t \in G} \left( \sum_{g' \in (t^{-1}Kt) \cap \stab{G}{Y}}\chi(p(g'))\right) \\
 										 &= \frac{1}{\cardinality{K}\cdot \cardinality{\stab{G}{Y}}}\sum_{t \in G} \cardinality{\ker(\restricted{p}{(t^{-1}Kt) \cap \stab{G}{Y}})} \left( \sum_{h \in p((t^{-1}Kt) \cap \stab{G}{Y})}\chi(h)\right) \\
 										 &= \frac{1}{\cardinality{K}\cdot \cardinality{\stab{G}{Y}}}\sum_{t \in G} \cardinality{\ker(\restricted{p}{(t^{-1}Kt) \cap \stab{G}{Y}})} \left( \sum_{h' \in J_t}\chi_t(h')\right) \\
 										 &= \sum_{t \in G} c_t \characterSum{\chi_t}{J_t},
			\end{aligned}
		\end{multline*}
		where $(1)$ follows from \autoref{proposition:induced_character_formula}, and
		\[c_t = \frac{\cardinality{\ker(\restricted{p}{(t^{-1}Kt) \cap \stab{G}{Y}})} \cdot \cardinality{p((t^{-1}Kt) \cap \stab{G}{Y})}}{\cardinality{K}\cdot \cardinality{\stab{G}{Y}}} = \frac{\cardinality{(t^{-1}Kt) \cap \stab{G}{Y}}}{\cardinality{K} \cdot \cardinality{\stab{G}{Y}}} > 0.\]
		Therefore, since ${\characterSum{\chi'}{J_t} \ge 0}$ for every $t \in G$,
		\[\characterSum{\hcLikeInduction{\chi}{H}{G}}{K} = 0 \iff \forall t \in G: \characterSum{\chi}{J_t} = 0.\]
	\end{proof}

	\begin{lemma} \label{lemma:stabilizer-character-permutaitons-local-transitivity}
		Continuing with \autoref{notation:section2}, let $G$ be a finite group $G$ acting on some set $X$, let $Y \subseteq X$ be some subset, and let $\chi \in \characters{H}$. Assume: 
		\begin{enumerate}
			\item $Y\supseteq Y_0=\{a_1 \dots a_\ell\}$
			\item\label{lemma:stabilizer-character-permutaitons-local-transitivity-h-assumption} \textCMarksSubgroupsThatAreNotxkTransitive{$\chi$}{$H$}{a}{\ell}.
		\end{enumerate}
	
		Let $K \subgroup G$ such that \textChiVanishesOnK{\hcLikeInduction{\chi}{H}{G}}{K}. Then for every $t \in G$, $K \cap \stab{G}{tY}$ is \xkGTransitive{ta}{\ell}{$\stab{G}{tY}$}.
	\end{lemma}
	\begin{proof}
		Assume that \textChiVanishesOnK{\hcLikeInduction{\chi}{H}{G}}{K}. By \autoref{lemma:stabilizer-character}, for every $t \in G$, ${J_t = p_t(K \cap \stab{G}{tY})}$ satisfies $\chiDoesntContainTrivial{\chi_t}{J_t}$.
		
		Let $t \in G$. It is easy to see that for every $g \in H$, $f_t(g)tx = tgx$. For every $h \in H_t$ there exists some $g \in H$ such that $f_t(g) = h$. Since $\chiKDoesntContainTrivial{\chi}{J_e}$, by assumption \ref{lemma:stabilizer-character-permutaitons-local-transitivity-h-assumption} of the lemma, $J_{\unit}$ is \xkGTransitive{a}{\ell}{$H$}. As such, there exists some $g' \in J_{\unit}$ such that $ga_i = g'a_i$ for every $1 \le i \le \ell$. 
		
		If we denote $h' \define f_t(g')$, we get 
		\begin{multline*}
			\begin{aligned}
				h'(ta_i) = f_t(g')t(a_i) = tg'(a_i) = tg(a_i) = f_t(g)t(a_i) = h(ta_i)
			\end{aligned}
		\end{multline*}
		for every $1 \le i \le \ell$. As such, $J_t$ is \xkGTransitive{ta}{\ell}{$H_t$}. We pull it back by $p_t$ (since $p_t$ is the quotient by the kernel of the action) to show that $K \cap \stab{G}{tY}$ is \xkGTransitive{ta}{\ell}{$\stab{G}{tY}$}.
	\end{proof}

	\begin{proof}[Proof of \autoref{lemma:stabilizer-character-permutaitons-global-transitivity}]
		Let $t_0 \in G$, and let $b_i \define t_0 a_i$ for all $1 \le i \le \ell - 1$. Define $g_i \in K$ by induction, for $0 \le i \le \ell-1$, so that $\forall j \le i \colon g_i(a_j) = b_j$.
		
		\textit{Base case -- $i = 0$}: define $g_0 = id$.
		
		\textit{Induction step}: assume that $g_j \in K$ are defined for all $j < i$. We define the following elements of $G$:
		
		\begin{description}[leftmargin=1cm, style=nextline]
			\item[$\bm{t'}$] 
			By assumption \ref{lemma:stabilizer-character-permutaitons-global-transitivity-covering} of \autoref{lemma:stabilizer-character-permutaitons-global-transitivity}, with $i+1 \le \ell$, there exists some $t' \in G$ such that
				\[{g_{i-1}(a_1)=b_1, \dots, g_{i-1}(a_{i-1})=b_{i-1}, g_{i-1}(a_i), b_i}\]
				all belong to $t'Y$.
			\item[$\bm{s_1}$]
		Since $t'^{-1}g_{i-1}(a_j) \in Y$ for all $j \le i$, assumption \ref{lemma:stabilizer-character-permutaitons-global-transitivity-permutation} of \autoref{lemma:stabilizer-character-permutaitons-global-transitivity} implies that there exists some ${s_1 \in \stab{G}{Y}}$, such that $s_1(a_j) = t'^{-1}g_{i-1}(a_j)$ for all $j \le i$. Equivalently, $t's_1(a_j) = g_{i-1}(a_j)$, for all $j \le i$.

			\item[$\bm{t}$]
			Denote $t \define t's_1$. Since $tY = t'Y$, $t^{-1}t'Y = Y$. 
			
			\item[$\bm{s_2}$]
			By assumption \ref{lemma:stabilizer-character-permutaitons-global-transitivity-permutation} of \autoref{lemma:stabilizer-character-permutaitons-global-transitivity} with $t^{-1}(b_j) \in Y$ for $j= 1, \dots, i$, we obtain some $s_2 \in \stab{G}{Y}$, such that $s_2(a_j) = t^{-1}(b_j)$ for all $j \le i$.

			\item[$\bm{s_3}$] Denote $s_3 \define ts_2t^{-1} \in \stab{G}{tY}$. For every $j \le i$,		
			\[s_3(g_{i-1}(a_j)) = ts_2t^{-1}t(a_j) = ts_2(a_j) = b_j.\]

			\item[$\bm{h}$] Since, by \autoref{lemma:stabilizer-character-permutaitons-local-transitivity}, $K \cap \stab{G}{tY}$ is \xkGTransitive{ta}{\ell}{$\stab{G}{tY}$} and ${t(a_j) = g_{i-1}(a_j)}$ for all $j \le i$, there exists some $h \in K$ such that for every $j \le i$, ${h(g_{i-1}(a_j)) = s_3(g_{i-1}(a_j)) = b_j}$.
		\end{description}

		 We can now define $g_i \define hg_{i-1} \in K$. Then: $\forall j \le i: g_i(a_j) = b_j$. By letting $g \define g_{\ell-1} \in K$, we deduce that $g(a_i) = b_i$ for all $1 \le i \le \ell-1$.
	\end{proof}

	\section{Characters of $\GL{n}{q}$} \label{section:glnq_characters}
	\paragraph{ } This section introduces several facts about the characters of $\GL{n}{q}$, and uses them to define and understand the main characters we need in the paper. We follow the terminology of \cite{gl_combinatorics}.
	
	\begin{definition} \label{definition:parabolic_group_characters_and_circ}
		Let $\alpha = (\alpha_1, \alpha_2, \dots, \alpha_s)$, such that $\alpha_1, \dots, \alpha_s \ge 1$ and ${\sum_{i=1}^s \alpha_i = n}$. Define the parabolic subgroup
		\[
		P_\alpha \define
		\begin{bmatrix}
		\GL{\alpha_1}{q}	& *					& \dots		& *					\\
		0					& \GL{\alpha_2}{q}	& \dots		& *					\\
		\vdots				& \vdots			& \ddots	& \vdots			\\
		0					& 0					& \dots		& \GL{\alpha_s}{q}	\\
		\end{bmatrix} 
		\subgroup \GL{n}{q}.
		\]
	
		Let $\chi_i\in\characters{\GL{\alpha_i}{q}}$. Define the characters
		\[(\chi_1 \times \chi_2 \times \dots \times \chi_s)
		\left(
		\begin{bmatrix}
		A_1		& *		 & \dots & *		\\
		0		& A_2	 & \dots & *		\\
		\vdots	& \vdots & \ddots& \vdots	\\
		0		& 0		 & \dots & A_s		\\
		\end{bmatrix} 
		\right)
		\define \prod_{i=1}^s(\chi_i(A_i))
		,\]
		
		\[\defineNotation{\chi_1 \circ \chi_2 \circ \dots \circ \chi_s} \define \induced{\chi_1 \times \chi_2 \times \dots \times \chi_s}{P_\alpha}{\GL{n}{q}} \in \characters{\GL{n}{q}}. \]
	\end{definition}

	\begin{proposition}[{\cite[Lemma 2.5]{green_gl_characters}}]
		The operator $\circ$ is bilinear, associative and commutative.
	\end{proposition}

	\paragraph{Remark} Let $1 \le \ell \le n$, and let $q$ be a prime power. Then, in the notation of \autoref{section:induced_characters_of_stab}, for every character $\chi$ of $\PGL{\ell}{q}$ we have ${\hcLikeInduction{\chi}{\PGL{\ell}{q}}{\GL{n}{q}} = \chi' \circ \trivialCharacter{n-\ell}}$ where $\chi'$ is the pull-back of $\chi$ to $\GL{\ell}{q}$.
	
	\paragraph{ } Using the operator $\circ$, we can define a family of characters.
	
	\begin{definition}
		Let ${\lambda = (\lambda_1, \dots, \lambda_\ell) \vdash n}$ be a partition of $n$. Define
		\[\defineNotation{\trivialCharacter{\lambda}} \define \trivialCharacter{\lambda_1} \circ \dots \circ \trivialCharacter{\lambda_\ell} \in \characters{\GL{n}{q}}.\]
	\end{definition}

	This family of characters is related to the action of $\GL{n}{q}$ on $\field{q}^n$. This can be seen as follows.
	
	For a sequence of integers $0 = d_0 \le d_1 \le \dots \le d_\ell = n$, we say that the sequence of subspaces ${0 = V_0 \subseteq V_1 \subseteq \dots \subseteq V_\ell = \field{q}^n}$ is a \defineTextNotation{\textit{flag of signature ${(d_1, \dots, d_\ell)}$}} if ${\dim(V_i) = d_i}$ for all ${i \le \ell}$.
	
	For a partition ${\lambda = (\lambda_1, \dots, \lambda_\ell) \vdash n}$ we associate the increasing sequence $d_i(\lambda) \define \sum_{j=1}^{i}\lambda_j$.

	\begin{example}	\label{example:trivial_character_lambda}
		We present examples for the characters $\trivialCharacter{\lambda}$. Since $\circ$ is commutative, we may permute the entries of $\lambda$. For example, $\trivialCharacter{(1, n-1)} = \trivialCharacter{(n-1, 1)}$.
		\begin{enumerate}
			\item The character $\trivialCharacter{(n)}$ is simply the trivial character, since the only flag of signature $(n)$ is $0 \subseteq \field{q}^n$.
			\item The value $\trivialCharacter{(1, n-1)}(g)$ is the number of 1-dimensional subspaces of $\field{q}^n$ that $g$ fixes, or equivalently, the number of fixed points of $g$ in the action of $\GL{n}{q}$ on $\projective^{n-1}\field{q}$.
			\item The value $\trivialCharacter{(1, 1, n-2)}(g)$ is the number of flags $0 \subseteq V_1 \subseteq V_2 \subseteq \field{q}^n$ that $g$ fixes, where $\dim(V_1) = 1$ and $\dim(V_2) = 2$.
			\item By \autoref{remark:permutation-representation}, $\trivialCharacter{\lambda}$ is the character of the permutation representation associated with the action of $\GL{n}{q}$ on flags of signature ${(d_1(\lambda), \dots, d_\ell(\lambda))}$. In other words, $\trivialCharacter{\lambda}(g)$ is the number of flags ${0 = V_0 \subseteq V_1 \subseteq \dots \subseteq V_\ell = \field{q}^n}$ of signature ${(d_1(\lambda), \dots, d_\ell(\lambda))}$ such that $gV_i = V_i$ for all $i \le \ell$.
		\end{enumerate}
	\end{example}

	The characters $\trivialCharacter{\lambda}$ are not irreducible in general. The following lemma decomposes them to irreducible characters.

	\begin{lemma}[{\cite[Lemma 2.4]{gl_combinatorics}}] \label{lemma:induced-trivial-character-split}
		Let $\lambda = (\lambda_1, \dots, \lambda_l) \vdash n$. Then
		\[\trivialCharacter{\lambda} = \sum_{\mu \vdash n}K_{\mu \lambda}\unipotentCharacter{\mu},\]
		where $K_{\mu \lambda}$ are the Kostka numbers, and $\unipotentCharacter{\mu}$ are some irreducible characters\declareFootnote{unipotent_ambiguity}{In \cite{gl_combinatorics}, these characters are labeled $\unipotentCharacter{\mu'^{(1)}}$. We define: $\unipotentCharacter{\mu} = \unipotentCharacter{\mu'^{(1)}}$, which is more similar to the notation of \cite{green_gl_characters}.}.
	\end{lemma}

	We define the unipotent characters ${\{\defineNotation{\unipotentCharacter{\mu}} \colon \mu \vdash n \}}$ as these irreducible summands of ${\{\trivialCharacter{\lambda} \colon \lambda \vdash n\}}$.

	\paragraph{ } In the remainder of the section, we study the characters $\transitivityCharacters$ of $\GL{n}{q}$, defined below. We decompose them to irreducible characters, and characterize the properties of subgroups of $\GL{n}{q}$ that they mark.
	\begin{definition} \label{definition:upper-triangle-characters}
		Define the following characters:
		\begin{multline*}
			\begin{aligned}
				\defineNotation{\transitivityCharacter} \define &\inducedTrivial{P_{(n-1, 1)}}{\GL{n}{q}}-\trivialCharacter{\GL{n}{q}} = \trivialCharacter{(n-1, 1)}-\trivialCharacter{(n)} & (n \ge 2)\\
				\defineNotation{\twoTransitivityCharacter} \define &\inducedTrivial{P_{(n-2, 1, 1)}}{\GL{n}{q}}-\trivialCharacter{\GL{n}{q}} = \trivialCharacter{(n-2, 1, 1)}-\trivialCharacter{(n)} & (n \ge 3)
			\end{aligned}
		\end{multline*}
	\end{definition}

	Using \autoref{lemma:induced-trivial-character-split}, we deduce:
	\begin{corollary} \label{corollary:parabolic-character-decomposition} 
		For the values of $n$ where the characters are defined:
		\begin{enumerate}
			\item $\trivialCharacter{\GL{n}{q}} = \trivialCharacter{(n)} = \unipotentCharacter{(n)}$
			\item $\transitivityCharacter = \unipotentCharacter{(n-1, 1)}$
			\item \label{corollary:parabolic-character-decomposition-2} $\twoTransitivityCharacter = 2\unipotentCharacter{(n-1, 1)}+\unipotentCharacter{(n-2, 1, 1)}+\unipotentCharacter{(n-2, 2)}$
		\end{enumerate}
		In \ref{corollary:parabolic-character-decomposition-2}, the term $\unipotentCharacter{(n-2, 2)}$ is omitted if $n = 3$.
	\end{corollary}
	
	\begin{proposition} \label{proposition:parabolic-character-sum-criteria}
		Let $K \subgroup \GL{n}{q}$, and let $V \cong \field{q}^n$ be the vector space on which $\GL{n}{q}$ acts. Consider the action of $\GL{n}{q}$ on $\projective(V)$, and denote by $\projectiveSemiPairs{V}$ the collection of flags of signature $(1, 2, n)$, that is,
		\[\defineNotation{\projectiveSemiPairs{V}} \define \left\{(0, V_1, V_2, V) \colon 0 \subseteq V_1 \subseteq V_2 \subseteq V, \dim(V_1) = 1, \dim(V_2) = 2\right\}.\]
		Then:
		\begin{enumerate}
			\item \label{proposition:parabolic-character-sum-criteria-1} $\chiDoesntContainTrivial{\transitivityCharacter}{K} \iff$ $K$ acts transitively on $\projective(V)$
			\item \label{proposition:parabolic-character-sum-criteria-2} $\chiDoesntContainTrivial{\twoTransitivityCharacter}{K} \iff$ $K$ acts transitively on $\projectiveSemiPairs{V} \iff $ $\forall (u, v), (u', v') \in V \times V$, both pairs linearly independent, there exist $A \in K$ and elements of $\field{q}$ -- $\alpha \ne 0, \beta, \gamma \ne 0$ -- such that $Au = \alpha u', Av = \beta u' + \gamma v'$.
		\end{enumerate}
	\end{proposition}

	\begin{proof}
		The proposition follows from \autoref{corollary:permutation-character-marks-transitive} and \autoref{example:trivial_character_lambda}.
	\end{proof}
	
	\section{The Maximal Subgroups of $\GL{n}{q}$} \label{section:maximal_subgroups_of_glnq}
	\paragraph{ }
	In this section we study the maximal subgroups of $\GL{n}{q}$ that do not contain $\SL{n}{q}$, and find a character $\chi$ such that \textChiDoesntVanishOnK{\chi}{K} for every maximal subgroup $K$.
	
	It is enough to study maximal subgroups of $\GL{n}{q}$, as \textChiDoesntVanishOnK{\chi}{K} implies \textChiDoesntVanishOnK{\chi}{K'} for every $K' \le K$. This can be easily seen, as \textChiDoesntVanishOnK{\chi}{K} implies that $\restricted{\chi}{K} = \trivialCharacter{K} + \chi'$ for some character $\chi'$ of $K$, which implies that $\restricted{\chi}{K'} = \trivialCharacter{K'} + \restricted{\chi'}{K'}$.
	
	\autoref{section:geometric_subgroups_of_glnq_description} describes the geometric maximal subgroups, \autoref{section:geometric_subgroups_of_glnq_restriction} checks whether $\chiKContainsTrivial{\transitivityCharacters}{K}$ for $K$ which is one of the geometric subgroups, and \autoref{section:non_geometric_subgroups} deals with the rest of the maximal subgroups, for the case of $n = 3$. For the main results, we only need these results for $n=2$ and $n=3$, but we add the general case since it is very similar to our specialized case.
	
	\paragraph{ } Before we proceed to describing the geometric groups, we make several definitions.
	
	 Let $G$ be a finite group, and let $d \in \positiveIntegers$. We denote by $\defineNotation{dG}$ some group extension of $G$ with $\subgroupIndex{dG}{G} = d$. There may exist many non-isomorphic groups that satisfy this condition, but we only consider properties that are common to all of the extensions.
	
	For $K \le \GL{m}{q}$, the wreath product $\defineNotation{K \wr \symmetric{\ell}}$, can be realized the group of matrices with $\ell \times \ell$ blocks, where:
	\begin{enumerate}
		\item Each block is of size $m \times m$.
		\item The non-zero blocks are elements of $K$.
		\item There is exactly one non-zero block in every row.
		\item There is exactly one non-zero block in every column.
	\end{enumerate}

	\paragraph{ }
	We denote by $\gal{\field{q}}{\field{q_0}}$ the Galois group of the extension $\factor{\field{q}}{\field{q_0}}$ where $q$ is a power of $q_0$.
	
	We use the nonstandard notation of $\defineNotation{\GammaLExt{d}{q}{q_0}}$ to denote the group of semi-linear morphisms on $\field{q}^d$ over $\field{q_0}$, that is, morphisms $f \colon \field{q}^d \to \field{q}^d$ that satisfy:
	\begin{enumerate}
		\item For every two vectors $u, v \in \field{q}^d$, $f(u+v) = f(u)+f(v)$.
		\item There exists some automorphism $\sigma \in \gal{\field{q}}{\field{q_0}}$ such that for every $v \in \field{q}^d$, ${f(\lambda u) = \sigma(\lambda) f(u)}$.
	\end{enumerate}
	
	The group $\GammaLExt{d}{q}{q_0}$ can be realized as the semidirect product ${\GL{d}{q} \rtimes \gal{\field{q}}{\field{q_0}}}$, which is the set 
	\[ \{ (A, \sigma) \colon A \in \GL{d}{q}, \sigma \in \gal{\field{q}}{\field{q_0}} \}. \]
	with multiplication defined by $(A, \sigma)(A', \sigma') = (A \sigma(A'), \sigma \sigma')$ where $\gal{\field{q}}{\field{q_0}}$ acts on every matrix entry by entry. The action of $\GammaLExt{d}{q}{q_0}$ on $\field{q}^d$ is defined by $(A, \sigma)(x) = A \sigma(x)$ where $\gal{\field{q}}{\field{q_0}}$ acts on every vector coordinate by coordinate.
	
	When $q_0$ is prime, we write ${\defineNotation{\GammaL{d}{q}} \define \GammaLExt{d}{q}{q_0}}$. Similarly, we denote ${\defineNotation{\aut{\field{q}}} \define \gal{\field{q}}{\field{q_0}}}$.

	\subsection{The Geometric Maximal Subgroups} \label{section:geometric_subgroups_of_glnq_description}
	
	\paragraph{Aschbacher's Theorem} To understand the structure of subgroups of $\GL{n}{q}$, we use a variation of a theorem due to Aschbacher\declareFootnote{aschbacher_references}{See \cite{aschbacher_original} for Aschbacher's original paper, or \cite{subgroup_structure_classical_groups} for a more detailed introduction.}. Roughly speaking, given a finite simple classical group $G_0$ and an extension $G$ of $G_0$ which satisfies ${G_0 \subgroup G \subgroup \aut{G_0}}$, Aschbacher's Theorem describes classes of subgroups  $\classicalFamily{1}{G}, \dots, \classicalFamily{8}{G}$ called "geometric classes" and an exceptional class $\classicalExceptions(G)$. The theorem states that every subgroup of $G$ which does not contain $G_0$ and is maximal among all such subgroups is a member of one of the classes $\classicalFamily{1}{G}, \dots, \classicalFamily{8}{G}, \classicalExceptions(G)$.
	
	The theorem can be generalized to several other families of groups, such as $\GL{n}{q}$. To simplify the definitions and theorems we do not state the general theorem, but restrict to the case of $G = \GL{n}{q}$. We use \cite[Theorem 1.2.1]{subgroup_structure_classical_groups} and the description of the geometric maximal subgroups in \cite[Chapter 4]{subgroup_structure_classical_groups} for the case of $\GL{n}{q}$. While the classes are defined by their geometric meaning, we divide them further by their structure. We present the classes in the following theorem, followed by brief descriptions of the classes.
	
	We note that while these classes may intersect and may contain subgroups that are not maximal among the subgroups of $G$ not containing $G_0$.
	
	\begin{theorem}
		Let $G=\GL{n}{q}$, and let $V \cong \field{q}^n$ be the space upon which $G$ acts with the usual action. Consider the families $\defineNotation{\classicalFamily{i}{\GL{n}{q}}}$ given in \autoref{table:classical-families}. Then, every $K \subgroup G$ such that $K \not \supergroup \SL{n}{q}$ is either contained in a member of $\classicalFamily{i}{G}$ for some $1 \le i \le 8$, or $K \in \classicalExceptions(G)$.

		\begin{table}[]
			\centering
			\caption{The families $\classicalFamily{i}{\GL{n}{q}}$ \label{table:classical-families}}
			\begin{threeparttable}
				\begin{tabular}{CCcC}
					\hline
					\classicalFamilySymbol{i}&Type
					&Description 
					&Conditions  \\ \hline \hline
					\classicalFamilySymbol{1}&P_{(m, n-m)}
					&\makecell{Stabilizer of $W \subseteq V$ \\ with $\dim(W) = m$} 
					& 1 \le m \le n-1 \\ \hline
					\classicalFamilySymbol{2}&\GL{m}{q} \wr \symmetric{t}
					&\makecell{Permutations of decompositions \\ of $V$ to $t$ $m$-dimensional spaces}
					&\makecell{n = mt \\ t \ge 2}\\ \hline
					\classicalFamilySymbol{3}& \GammaLExt{m}{q^r}{q}
					&\makecell{Stabilizers of the structure of $V$ \\ as a vector space over $\field{q^r}$}
					&\makecell{n = mr}\\ \hline
					\classicalFamilySymbol{4}&\GL{n_1}{q} \otimes \GL{n_2}{q}
					&\makecell{Stabilizers of the structure of $V$ \\ as $V \cong V_1 \otimes V_2$} 
					&\makecell{n=n_1n_2 \\ 2 \le n_1 < n_2} \\ \hline
					\classicalFamilySymbol{5}&\GL{n}{q'}
					&\makecell{Matrices with coefficients \\ in $\field{q'}$} 
					&\field{q'} \subseteq \field{q} \\ \hline
					\classicalFamilySymbol{6}&\normalizer{R}{\GL{n}{q}}
					& See \autoref{definition:c6_description}.
					&\makecell{r \text{ is prime} \\ n = r^m, r \divides q-1\\  \text{if } r = 2, 4 \divides q-1} \\ \hline
					\classicalFamilySymbol{7}&\GL{m}{q} \wr \symmetric{t}
					&\makecell{Permutations of decompositions \\ of $V = V_1 \otimes \dots \otimes V_t$ where \\ $dim V_i = m$}
					&\makecell{n=m^t \\ m \ge 3}\\ \hline
					\classicalFamilySymbol{8}&\SP{n}{q}
					& See \autoref{definition:c8_description}.
					&4 \le n \text{ is even}  \\ \hline
					\classicalFamilySymbol{8}&U_n(q^{\frac{1}{2}})
					& See \autoref{definition:c8_description}.
					&\makecell{q \text{ is a square} \\ n \ge 3}  \\ \hline
					\classicalFamilySymbol{8}&O_n^\epsilon(q)
					& See \autoref{definition:c8_description}.
					&\makecell{q \text{ is odd} \\ n \ge 3}\\
				\end{tabular}
			\end{threeparttable}
		\end{table}
	\end{theorem}

	\paragraph{Remark} For subgroups $\SL{n}{q} \subgroup G' \subgroup \GL{n}{q}$, the families $\defineNotation{\classicalFamily{i}{G'}}$ are defined as ${\{G' \cap C \colon C \in \classicalFamily{i}{\GL{n}{q}}\}}$, and likewise for $\classicalExceptions$.

	\paragraph{Description of the geometric classes} Let $G=\GL{n}{q}$, and let $V \cong \field{q}^n$ be the space upon which $G$ acts with the usual action.
	
	\begin{description}[leftmargin=1.4cm, style=nextline]
		\item[$\mathbf{\classicalFamily{1}{G}}$]
		This class contains $\stab{G}{W}$, which are the stabilizers of subspaces ${W \subseteq V}$. Up to conjugation, these are the groups $P_{(m, n-m)}$ defined in \autoref{definition:parabolic_group_characters_and_circ}, with $m = \dim(W)$.

		\item[$\mathbf{\classicalFamily{2}{G}}$] 
		Write $V = \bigoplus_{i=1}^t V_i$, where each $V_i$ is of dimension $m$. Consider the subgroup of $G$ that stabilizes each of the vector spaces. It can be extended by allowing permutations of the vector spaces, to get a group conjugate to $\GL{m}{q} \wr \symmetric{t}$.
		
		\item[$\mathbf{\classicalFamily{3}{G}}$] 
		If $n = r m$, we can consider $V$ as a vector space of dimension $m$ over $\field{q^r}$. Then, the semi-linear group $\GammaLExt{m}{q^r}{q}$ naturally acts linearly (over $\field{q}$) on $V$ and embeds as a proper\declareFootnote{gl22_exception}{A single exceptional case occurs when $G = \GL{2}{2}$. Then, the subgroup $\GammaLExt{1}{4}{2}$ is not proper, so we define $\classicalFamily{3}{G}$ as the singleton $\{\GL{1}{4}\}$.} subgroup of $\GL{n}{q}$.
		
		\item[$\mathbf{\classicalFamily{4}{G}}$] 
		This class contains the stabilizers of decompositions of the form ${V \cong V_1 \otimes V_2}$.
		
		\item[$\mathbf{\classicalFamily{5}{G}}$] 
		Considering $G$ as a group of matrices, this class contains the subgroups of matrices $\GL{n}{\field{q'}}$, where we restrict the coefficients to some subfield $\field{q'}$ of $\field{q}$.
		
		\item[$\mathbf{\classicalFamily{6}{G}}$] 
		See \autoref{definition:c6_description}.
		
		\item[$\mathbf{\classicalFamily{7}{G}}$] 
		If we write $V = \bigotimes_{i=1}^t V_i$, where each $V_i$ is of dimension $m$, consider the subgroup of $G$ that stabilizes each of the vector spaces. It can be extended by allowing permutations of the vector spaces, to get a group isomorphic to $\GL{m}{q} \wr \symmetric{t}$.
		
		\item[$\mathbf{\classicalFamily{8}{G}}$] 
		See \autoref{definition:c8_description}.
		
	\end{description}

	\paragraph*{The class $\mathbf{\classicalExceptions(G)}$} We only describe the groups in this class for $n = 2, 3$, in \autoref{section:c6_and_s_for_n_eq_2} and \autoref{section:non_geometric_subgroups}. For more information about the general case, see \cite[Section 1.2]{subgroup_structure_classical_groups}.

	\begin{definition} (A full definition is given in {\cite[Section 4.6]{subgroup_structure_classical_groups}}) \label{definition:c6_description} 
		We define the groups in $\classicalFamily{6}{\GL{n}{q}}$. Assume that $n=r^m$, for a prime $r$ that does not divide $q$, and let
		\[
		\fieldInvertible{q} \ni \omega = 
		\begin{cases}
		\text{primitive $r$-root of unity} & \text{if }r \ne 2\\
		\text{primitive $4$-root of unity} & \text{if }r = 2
		\end{cases}.
		\]
		For some $R_i \subgroup \groupSpan{\omega} \wr \symmetric{r}$, $1 \le i \le m$, such that $R_i$ is not contained in the group of diagonal matrices (these are given explicitly in {\cite[Section 4.6]{subgroup_structure_classical_groups}}), let ${R' = R_1 \otimes R_2 \dots \otimes R_m}$ as in $\classicalFamilySymbol{7}$. Let $\omega I$ be the diagonal matrix with $\omega$ on the diagonal, and let $R = \groupSpan{\omega I, R'}$ with its action on ${\field{q}^r \otimes \field{q}^r \otimes \dots \otimes \field{q}^r \cong \field{q}^n}$. We define $\classicalFamily{6}{\GL{n}{q}}$ to be the collection of subgroups of $\GL{n}{q}$ that are conjugate to the normalizers of such groups $R$.
	\end{definition}

	\begin{definition} (A full definition is given in {\cite[Section 2]{subgroup_structure_classical_groups}}) \label{definition:c8_description}
		The class $\classicalFamily{8}{\GL{n}{q}}$ consists of the groups listed in the last three rows of \autoref{table:classical-families}. Each group has some bilinear form, such that the group is the subgroup of $\GL{n}{q}$ that preserves the bilinear form. We describe the bilinear form by some standard basis, defining the group up to conjugation. These are described in \autoref{table:classical-bases}.
		\begin{table}[]
			\centering
			\caption{The bilinear forms defining the subgroups in the class $\classicalFamily{8}{\GL{n}{q}}$ \label{table:classical-bases}}
			\begin{threeparttable}
				\begin{tabular}{CcCC}
					\hline
					K & Basis type
					& Basis
					& Relations \\ \hline \hline
					\makecell{\SP{n}{q} \\ n=2m} & Symplectic
					&\makecell{e_1,\dots, e_m \\ f_1, \dots, f_m}
					&\makecell{\innerProduct{e_i}{e_j}=0 \\ \innerProduct{f_i}{f_j}=0 \\ \innerProduct{e_i}{f_j}=\delta_{i j}}  \\ \hline
					
					\makecell{U_n(q^{\frac{1}{2}}) \\ n=2m} & Unitary
					&\makecell{e_1,\dots, e_m \\ f_1, \dots, f_m}
					&\makecell{\innerProduct{e_i}{e_j}=0 \\ \innerProduct{f_i}{f_j}=0 \\ \innerProduct{e_i}{f_j}=\delta_{i j}}  \\ \hline
					
					\makecell{U_n(q^{\frac{1}{2}}) \\ n=2m+1} & Unitary
					&\makecell{e_1,\dots, e_m \\ f_1, \dots, f_m \\ x}
					&\makecell{\innerProduct{e_i}{e_j}=0, \ \innerProduct{e_i}{x}=0 \\ 
						\innerProduct{f_i}{f_j}=0, \ \innerProduct{f_i}{x}=0 \\ 
						\innerProduct{x}{x}=1, \ \innerProduct{e_i}{f_j}=\delta_{i j}} \\ \hline
					
					\makecell{O_n^+(q) \\ n=2m} & Orthogonal
					&\makecell{e_1,\dots, e_m \\ f_1, \dots, f_m}
					&\makecell{\innerProduct{e_i}{e_i}=0 \\ \innerProduct{f_i}{f_i}=0 \\ \innerProduct{e_i}{f_j}=\delta_{i j}}  \\ \hline
					
					\makecell{O_n^-(q) \\ n=2m} & Orthogonal
					&\makecell{e_1,\dots, e_{m-1} \\ f_1, \dots, f_{m-1} \\ x, y}
					&\makecell{\innerProduct{e_i}{e_i}=0, \ \innerProduct{e_i}{x}=0, \ \innerProduct{e_i}{y}=0 \\ 
					\innerProduct{f_i}{f_i}=0, \ \innerProduct{f_i}{x}=0, \ \innerProduct{f_i}{y}=0 \\ 
					\innerProduct{e_i}{f_j}=\delta_{i j}, \ \innerProduct{x}{x} = 1 \\ \innerProduct{x}{y} = 1, \ \innerProduct{y}{y}\ne 0 }\\ \hline
					
					\makecell{O_n^\circ(q) \\ n=2m+1} & Orthogonal
					&\makecell{e_1,\dots, e_m \\ f_1, \dots, f_m \\ x}
					&\makecell{\innerProduct{e_i}{e_i}=0, \ \innerProduct{e_i}{x}=0 \\ 
						\innerProduct{f_i}{f_i}=0, \ \innerProduct{f_i}{x}=0 \\ 
						\innerProduct{x}{x}\ne 0, \ \innerProduct{e_i}{f_j}=\delta_{i j}} \\
				\end{tabular}
			\end{threeparttable}
		\end{table}
	\end{definition}

	\subsection{The Restriction of the Characters to the Geometric Subgroups} \label{section:geometric_subgroups_of_glnq_restriction}
	\paragraph{ }
	Recall \autoref{definition:upper-triangle-characters} where $\transitivityCharacter$ and $\twoTransitivityCharacter$ were introduced. In this section, we prove that for every subgroup $K \subgroup \GL{n}{q}$ which is a member of one of the geometric families, \textChiDoesntVanishOnK{\chi}{K} for at least one of the characters $\chi = \transitivityCharacters$. While $\transitivityCharacter$ is a summand of $\twoTransitivityCharacter$ for $n \ge 3$, $\twoTransitivityCharacter$ is not defined for $n = 2$. Therefore, whenever possible, we prove the result for $\transitivityCharacter$. 
	
	Throughout the section, we heavily use \autoref{proposition:parabolic-character-sum-criteria} without explicitly referring to it.
	
	\paragraph{($\bm{\transitivityCharacter}$)} In this part of the section, we prove that \textChiDoesntVanishOnK{\transitivityCharacter}{K} if $K$ is a member of one of the families $\classicalFamilySymbol{1}, \classicalFamilySymbol{2}, \classicalFamilySymbol{4}, \classicalFamilySymbol{5}$ or $\classicalFamilySymbol{7}$. It is sufficient to find $v, w \in \projective(V)$ such that no $k \in K$ satisfies $kv = w$.
	\begin{proposition} \label{proposition:trasitivity_c1}
		If $K \in \classicalFamily{1}{\GL{n}{q}}$ then $\chiKContainsTrivial{\transitivityCharacter}{K}$.
	\end{proposition}
	\begin{proof}
		Let $K \in \classicalFamily{1}{\GL{n}{q}}$. The group $K$ stabilizes some proper non-zero subspace $W$. Choose $v \in \projective W$ and $w \not \in \projective W$. Then, for every $k \in K$, $kv \ne w$.
	\end{proof}

	\begin{proposition}
		If $K \in \classicalFamily{2}{\GL{n}{q}}$ then $\chiKContainsTrivial{\transitivityCharacter}{K}$.
	\end{proposition}
	\begin{proof}
		Let $K \in \classicalFamily{2}{\GL{n}{q}}$. Consider the basis $\{e_1, \dots, e_n\}$ of $V$ such that the elements of $K$ are realized as block matrices (with respect to the decomposition $V = \bigoplus_{i=1}^t V_i$) over this basis. As a matrix with this basis, the first column of every element of $K$ is all zeros but a single block of length $m < n$. Therefore, there is no element $k \in K$ with $k \projectiveVector{e}_1 = \projectiveVector{\sum_{i=1}^n e_i}$.
	\end{proof}

	\begin{proposition}
		If $K \in \classicalFamily{4}{\GL{n}{q}}$ then $\chiKContainsTrivial{\transitivityCharacter}{K}$.
	\end{proposition}
	\begin{proof}
		Let $K \in \classicalFamily{4}{\GL{n}{q}}$. Recall that $K \cong \GL{n_1}{q} \otimes \GL{n_2}{q}$, and consider the decomposition ${V \cong V_1 \otimes V_2}$. The action of $K$ is defined as follows: $(A \otimes B)(u \otimes v) = Au \otimes Bv$. By denoting $\{e_1, e_2, \dots\}$ and $\{e_1', e_2', \dots\}$ as bases of $V_1$ and $V_2$ respectively, we see that no element in $K$ satisfies: $(A \otimes B)(\projectiveVector{e_1 \otimes e_1'}) = \projectiveVector{e_1 \otimes e_1' + e_2 \otimes e_2'}$.
	\end{proof}

	\begin{proposition}
		If $K \in \classicalFamily{5}{\GL{n}{q}}$ then $\chiKContainsTrivial{\transitivityCharacter}{K}$.
	\end{proposition}
	\begin{proof}
		Let $K \in \classicalFamily{5}{\GL{n}{q}}$. Choose $v \in \projective \left(\field{q'}^n\right), w \in \projective\left(\field{q'} \times (\field{q}-\field{q'}) ^ {n-1}\right)$. Then $kv \ne w$ for all $k \in K$.
	\end{proof}

	\begin{proposition}
		If $K \in \classicalFamily{7}{\GL{n}{q}}$ then $\chiKContainsTrivial{\transitivityCharacter}{K}$.
	\end{proposition}
	\begin{proof}
		Let $K \in \classicalFamily{7}{\GL{n}{q}}$. Then, an element of $K$ permutes the subspaces and acts on each of the subspaces in the decomposition: $V \cong V_1 \otimes \dots \otimes V_t$. Let $e^{(i)}_1, \dots, e^{(i)}_m$ the basis of $V_i$.
		
		Then, for every $k \in K$ there exist $v^{(i)} = \sum_{j=1}^m \lambda^{(i)}_j e^{(i)}_j \in V_i$ such that
		\[k \left(\projectiveVector{\bigotimes_{i = 1}^t e^{(i)}_1}\right) =
		\projectiveVector{\bigotimes_{i = 1}^t v^{(i)}} =
		\projectiveVector{\bigotimes_{i = 1}^t \sum_{j=1}^m \lambda^{(i)}_j e^{(i)}_j} =
		\projectiveVector{\sum_{j_1, \dots, j_t=1}^m \left(\prod_{i=1}^t \lambda^{(i)}_{j_i} \right) \cdot \left( \bigotimes_{i = 1}^t e^{(i)}_{j_i} \right)}.\]
		By comparing the coefficients, we see that no $k \in K$ can satisfy
		\[k \left(\projectiveVector{\bigotimes_{i = 1}^t e^{(i)}_1}\right) = \projectiveVector{\bigotimes_{i = 1}^t e^{(i)}_1 + \bigotimes_{i = 1}^t e^{(i)}_2},\]
		since this implies that both $\prod_{i=1}^t \lambda^{(i)}_1 = \prod_{i=1}^t \lambda^{(i)}_2 \ne 0$ and ${\lambda^{(1)}_1 \prod_{i=2}^t \lambda^{(i)}_2 = 0}$, which leads to a contradiction.
	\end{proof}

	\paragraph{($\bm{\twoTransitivityCharacter}$)} In this part of the section, we prove that \textChiDoesntVanishOnK{\twoTransitivityCharacter}{K}, for $K$ in one of the families $\classicalFamilySymbol{3}, \classicalFamilySymbol{6}, \classicalFamilySymbol{8}$. Since $\twoTransitivityCharacter$ is only defined for $n \ge 3$, we assume this inequality holds.
	
	\begin{proposition}
		If $K \in \classicalFamily{3}{\GL{n}{q}}$ then $\chiKContainsTrivial{\twoTransitivityCharacter}{K}$.
	\end{proposition}
	\begin{proof}
		Assume otherwise. Then, for every three vectors $u, v, w \in \field{q}^n$ such that $u,v$ and $u, w$ are linearly independent, there exists some $k \in K$ and $\lambda, \mu, \nu \in \field{q}$ with $\lambda, \nu \ne 0$ such that $k u = \lambda u$ and $k v = \mu u + \nu w$.
		
		Let $\alpha \in \field{q^r}$ be such that $\field{q^r} = \field{q}(\alpha)$, and let $\{b_1, \dots, b_m\}$ be a basis of $V \cong \field{q^r}^m$. Define $e_{i+r(j-1)} = \alpha^{i-1} b_j$ for $i = 1, \dots r$ and $j = 1, \dots, m$. Then $\{e_1, \dots, e_n\}$ is a basis for $V \cong \field{q}^n$.
		
		For every $k = (A, \sigma) \in K = \GammaLExt{m}{q^r}{q}$ such that $k e_1 = \lambda e_1, \lambda \ne 0$, we have $k e_2 = k (\alpha e_1) = \sigma(\alpha) \lambda e_1 = \sigma(\alpha) \lambda b_1$.
		
		Assume that $m \ge 2$. Choose $u = e_1$, $v = e_2$ and $w = e_{r+1}$. Then, there exists some $k \in K$ such that ${ke_1 = \lambda e_1}$ and ${ke_2 = \mu e_1 + \nu e_{r+1} = \mu b_1 + \nu b_2}$ for $\lambda, \mu, \nu \in \field{q}$ with $\lambda, \nu \ne 0$. That implies $\sigma(\alpha) \lambda b_1 = \mu b_1 + \nu b_2$ which leads to a contradiction. Thus, we may assume that $m = 1$ and $r = n \ge 3$.
		
		We note that $\sigma(\alpha)$ cannot have an irreducible polynomial of degree $2$ over $\field{q}$, since then $\alpha$ would have such a polynomial as well, which contradicts $r \ge 3$.
		
		Choose $u = e_1$, $v = e_2$ and $w = \sigma(\alpha)^2 e_1$. Since $r \ge 3$, $\sigma(\alpha)^2 \not \in \field{q}$ (otherwise $\sigma(\alpha)$ would have an irreducible polynomial of degree $2$ over $\field{q}$), so $u$ and $w$ are linearly independent over $\field{q}$. Then, there exists some $k \in K$ such that ${ke_1 = \lambda e_1}$ and ${k e_2 = \mu e_1 + \nu \sigma(\alpha)^2 e_1}$, for ${\lambda, \mu, \nu \in \field{q}}$ with ${\lambda, \nu \ne 0}$. Then, by comparison of coefficients of $e_1$ in the equality ${\sigma(\alpha) \lambda e_1 = k e_2 = \mu e_1 + \nu \sigma(\alpha)^2 e_1}$ we deduce that $\sigma(\alpha)$ has an irreducible polynomial of degree $2$ over $\field{q}$, which leads to a contradiction.
	\end{proof}

	\begin{proposition}
		If $K \in \classicalFamily{6}{\GL{n}{q}}$ then $\chiKContainsTrivial{\twoTransitivityCharacter}{K}$.
	\end{proposition}
	\begin{proof}
		Let $K \in \classicalFamily{6}{\GL{n}{q}}$ be the normalizer of some $R$ as in \autoref{definition:c6_description}, and assume that $\chiKDoesntContainTrivial{\twoTransitivityCharacter}{K}$. Let $r_1 \in R_1$ be a non-diagonal element. Choose some $v_1 \in V_1$ which is not an eigenvalue of $r_1$, let $v_2 = r_1(v_1)$, and complete these to a basis of $V_1$. Let $B = \{b_1, b_2, \dots\}$ be a basis of $V_2 \otimes \dots \otimes V_m$. Let $u \in R$.
		
		Let $k \in K$. Then $k$ is in the normalizer of $R$. Thus, $kR = Rk$. By $\chiKDoesntContainTrivial{\twoTransitivityCharacter}{K}$, there exists $\lambda, \mu, \nu \in \field{q}$ with $\lambda, \nu \ne 0$ such that:
		\begin{multline*}
			\begin{aligned}
				& k (v_1\otimes b_1) = \lambda v_1\otimes b_1 \text{ and} \\
				& k (v_2\otimes b_1) = \mu v_1\otimes b_1 + \nu (v_2\otimes b_1 + u).
			\end{aligned}
		\end{multline*}
		Then,
		\begin{multline*}
			\begin{aligned}
				& Rk (v_1\otimes b_1) = \{\lambda h (v_1\otimes b_1) \colon h \in R\}, \\
				& kR (v_1\otimes b_1) = \{k h (v_1\otimes b_1) \colon h \in R\}, \\
				& Rk (v_2\otimes b_1) = \{\mu h (v_1\otimes b_1) + \nu h (v_2\otimes b_1) + \nu h (u) \colon h \in R\} \text{ and} \\
				& kR (v_2\otimes b_1) = \{k h (v_2\otimes b_1) \colon h \in R\} = \{k h (v_1\otimes b_1) \colon h \in R\}
			\end{aligned}
		\end{multline*}
		with the last equality calculated by variable change, with ${h' = h (r_1 \otimes I \otimes \dots \otimes I)^{-1}}$.
		
		Comparing the various lines, we deduce 
		\[ \{\lambda h (v_1\otimes b_1) \colon h \in R\} = \{\mu h (v_1\otimes b_1) + \nu h (v_2\otimes b_1) + \nu h (u) \colon h \in R\}. \]
		In particular, there exists some $h_0 \in R$ such that
		\[\lambda h_0 (v_1\otimes b_1) = \mu v_1\otimes b_1 + \nu v_2\otimes b_1 + \nu u. \]
		
		Denote $h_0 = \delta A \otimes C$ where $\delta \in \groupSpan{\omega}, A \in R_1$ and $C \in R_2 \otimes \dots \otimes R_m$. The coefficients of $v_2\otimes b_1$ satisfy
		\[\lambda \delta A_{2, 1} C_{1, 1} = \nu \implies A_{2, 1} \ne 0.\]
		Since $R_i \subgroup \groupSpan{\omega} \wr \symmetric{r}$, for every $s = 1$ or $3 \le s \le r$, $A_{s, 1} = 0$. By comparing the coefficients of $v_1\otimes b_1$, we get $\mu = \lambda \delta A_{1, 1} C_{1, 1} = 0$.
		
		We get $h_0 (v_1 \otimes b_1) = \delta A_{2, 1} v_2 \otimes C b_1$. To finish, it is sufficient to choose (if $r = 2$, $2^m = n \ge 3 \implies m \ge 2$)
		\[
			u = 
			\begin{cases}
			v_3 \otimes b_1 	& r \ge 3\\
			v_1 \otimes b_2,    & r = 2, m \ge 2
			\end{cases}
		\]
		We get a contradiction to $\nu u = \lambda \delta A_{2, 1} v_2 \otimes Cb_1-\nu v_2 \otimes b_1$.
	\end{proof}

	\begin{proposition} \label{proposition:trasitivity_c8}
		If $K \in \classicalFamily{8}{\GL{n}{q}}$ then $\chiKContainsTrivial{\twoTransitivityCharacter}{K}$.
	\end{proposition}
	\begin{proof}
		It is enough to find vectors $u, v, w$ such that $\innerProduct{u}{u} = \innerProduct{u}{v} = 0$ and $\innerProduct{u}{w} \ne 0$. With such vectors, for every $\lambda, \mu, \nu \in \field{q}$ with $\lambda, \nu \ne 0$, 
		\[\innerProduct{\lambda u}{\mu u+ \nu w} = \lambda \nu \innerProduct{u}{w} \ne 0 = \innerProduct{u}{v},\]
		so there is no $k\in K$ such that 
		$ku = \lambda u, kv = \mu u+ \nu w$.
		
		We define the vectors $u, v, w$ as $u = e_1$, $w = f_1$ and
		\[
			v = 
			\begin{cases}
				f_2 & K=\SP{n}{q}, U_{2m}(q^{\frac{1}{2}}) \text{ or } O_n^+(q) \\
				x & K=U_{2m+1}(q^{\frac{1}{2}}), O_n^-(q) \text{ or } O_n^\circ(q)
			\end{cases}.
		\]
	\end{proof}

	From propositions \ref{proposition:trasitivity_c1}-\ref{proposition:trasitivity_c8}, we conclude:
	\begin{corollary} \label{corollary:no-non-vanishing-subgroups-gt-3}
		If $n \ge 3$, \textCMarksSubgroups{$\twoTransitivityCharacter$}{$\GL{n}{q}$} that does not contain $\SL{n}{q}$ and does not belong to $\classicalExceptions(\GL{2}{q})$.
	\end{corollary}
	
	For $n=2$, $\classicalFamily{8}{\GL{2}{q}}$ is empty. Therefore, we can deduce:

	\begin{corollary} \label{corollary:no-non-vanishing-subgroups-2}
		The character \textCMarksSubgroups{$\transitivityCharacter$}{$\GL{2}{q}$} that does not contain $\SL{2}{q}$, and does not belong to $(\classicalFamilySymbol{3} \cup \classicalFamilySymbol{6} \cup \classicalExceptions)(\GL{2}{q})$.
	\end{corollary}

	\subsection{The Non-Geometric Subgroups for $n = 3$} \label{section:non_geometric_subgroups}
	\begin{theorem}[{\cite[Table 8.4]{maximal_subgroups}}] \label{theorem:classical_exceptions_of_sl3}
		$\classicalExceptions(\SL{3}{q})$ are given in \autoref{table:classical_exceptions_of_sl3}.
		\declareAbcCounter{classicalExceptionsCounter}
		\begin{table}[]
		\centering
			\caption{$\classicalExceptions(\SL{3}{q})$} \label{table:classical_exceptions_of_sl3}
			\begin{threeparttable}
				\begin{tabular}{cCCC}
					\hline
					Case & K' \le \SL{3}{q} & \text{Conditions on } q
					& \text{Order}\\ \hline \hline
					\abcCounter{classicalExceptionsCounter} & \makecell{d\hspace{0.1cm}\PSL{2}{7} \\ d = \gcd(q-1, 3)}
					&\makecell{q\text{ is prime} \\ q \equiv 1, 2, 4 \mod 7 \\ q \ne 2}
					& 168d  \\ \hline
					
					\abcCounter{classicalExceptionsCounter} \label{dsadsa}  & 3\alternating{6}
					&\makecell{q\text{ is prime} \\ q \equiv 1, 4 \mod 15}
					&1080  \\ \hline
					
					\abcCounter{classicalExceptionsCounter} & 3\alternating{6}
					&\makecell{q\text{ is the square of a prime }p \\ p \equiv 2, 3 \mod 5 \\ p \ne 3}
					&1080  \\ \hline
				\end{tabular}
			\end{threeparttable}
		\end{table}
	\end{theorem}

	By looking at the table, we see that there are finitely many possibilities for these groups.	To understand the transitivity of their actions, we use the following immediate consequence of the orbit-stabilizer theorem.
	
	Let $G$ be a finite group, acting transitively on some set $X$. Let $K$ be the kernel of the action, and let $x \in X$. Then 
	\begin{equation} \label{corollary:orbit_stabilizer_orders}
		\subgroupIndex{G}{K} = \subgroupIndex{G}{\stab{G}{x}} \subgroupIndex{\stab{G}{x}}{K} = \cardinality{X}\subgroupIndex{\stab{G}{x}}{K}.
	\end{equation}
	In particular, $\cardinality{X}$ divides $\subgroupIndex{G}{K}$ and $\cardinality{G}$.
	
	\paragraph{ } Recall that $\classicalExceptions(\SL{n}{q}) = \{C \cap \SL{n}{q} \colon C \in \classicalExceptions(\GL{n}{q})\}$. We use \eqref{corollary:orbit_stabilizer_orders} in the following lemma.
	\begin{lemma}
		$\forall K \in \classicalExceptions(\GL{3}{q}) \colon \chiKContainsTrivial{\twoTransitivityCharacter}{K}$.
	\end{lemma}
	\begin{proof}
		By \autoref{theorem:classical_exceptions_of_sl3}, the maximal subgroups of $\SL{3}{q}$ have orders bounded by $1080$.
		
		Assume $\chiKDoesntContainTrivial{\twoTransitivityCharacter}{K}$. Then, $K$ acts transitively on $X = \projectiveSemiPairs{V}$, a set of size $(q^2+q+1)(q+1)$. By \eqref{corollary:orbit_stabilizer_orders}, ${(q^2+q+1)(q+1) \divides \cardinality{K}}$.
		
		Denote $K' \define K \cap \SL{3}{q} \in \classicalExceptions(\SL{n}{q})$. By the second isomorphism theorem, 
		\[\subgroupIndex{K}{K'} = \subgroupIndex{\SL{3}{q} K}{\SL{3}{q}} \divides \subgroupIndex{\GL{3}{q}}{\SL{3}{q}} = q-1.\]
		Therefore, $\frac{\cardinality{K}}{c} \divides \cardinality{K'}$, with $c = \gcd(\cardinality{K}, q-1)$. Then, $\frac{(q^2+q+1)(q+1)}{c'} \divides \cardinality{K'}$, with $c' = \gcd(\cardinality{K}, q-1, (q^2+q+1)(q+1))$. This implies that ${\frac{(q^2+q+1)(q+1)}{q-1} \le \cardinality{K'} \le 1080}$. Since the rational function is increasing with $q$, $q < 32$. 
		 
		 Then, it is easy to verify that for none of the possibilities for $\cardinality{K'}$, $\frac{(q^2+q+1)(q+1)}{c'} \divides \cardinality{K'}$, except for the case where $q = 2$. However, $q=2$ is explicitly ruled out by the conditions for $K$. We deduce that \textChiDoesntVanishOnK{\twoTransitivityCharacter}{K}, for $K \in \classicalExceptions(\GL{3}{q})$, for every prime power $q$.
	\end{proof}

	We conclude:
	
	\begin{corollary} \label{corollary:no-non-vanishing-subgroups-3}
		The character \textCMarksSubgroupsThatDoNotContainH{$\twoTransitivityCharacter$}{$\GL{3}{q}$}{$\SL{3}{q}$}.
	\end{corollary}

	\section{Proofs of the Main Theorems for $n \ge 3$} \label{section:nonvanishing_characters}
	\paragraph{ } In this section, we use the results from the previous sections to deduce the main theorems for $n \ge 3$. Since our proof of \autoref{theorem:nonvanishing_characters_of_gln} for $\GL{n}{q}$ uses results about $\PGL{3}{q}$, we prove that \autoref{theorem:nonvanishing_characters_of_gln} for $\GL{n}{q}$ implies \autoref{theorem:nonvanishing_characters_of_gln} for $\PGL{n}{q}$ as well as \autoref{theorem:nonvanishing_characters_of_sln}, and only then prove \autoref{theorem:nonvanishing_characters_of_gln} for $\GL{n}{q}$.
	
	\subsection{The Special Linear Group} \label{section:special_linear_unipotent_restriction}
	\paragraph{ } In this section, we assume \autoref{theorem:nonvanishing_characters_of_gln} for $\GL{n}{q}$, and prove \autoref{theorem:nonvanishing_characters_of_sln} for $\SL{n}{q}$.
	
	\begin{theorem}[e.g. {\cite[Corollary 4.12]{sln_characters}}] \label{theorem:restriction_to_normal_subgroup}
		Let $G$ be a finite group, let $K \normalSubgroup G$ and let $\chi \in \irreducibleCharacters{G}$. Assume that $G = K C$, where $C \subgroup G$ is cyclic and $K \cap C = 1$. Then, if $\restricted{\chi}{K}$ splits into $b$ irreducible components, there are precisely $\frac{\cardinality{C}}{b}$ distinct irreducible characters of $G$ whose restriction to $K$ equals $\restricted{\chi}{K}$.
	\end{theorem}
	\paragraph{Remark} By {\cite[Proposition 4.11]{sln_characters}}, the restricted components are all of the same degree.

	\begin{corollary} \label{corollary:restriction_of_unipotent_characters_to_sl}
		Let $\unipotentCharacter{\mu} \in \irreducibleCharacters{\GL{n}{q}}$. Then, $\restricted{\unipotentCharacter{\mu}}{\SL{n}{q}} \in \irreducibleCharacters{\SL{n}{q}}$.
	\end{corollary}
	\begin{proof}
		Let $\theta \in \irreducibleCharacters{\fieldInvertible{q}}$, and consider the character ${\theta(\det(g))\unipotentCharacter{\mu}(g) \in \irreducibleCharacters{\GL{n}{q}}}$. For every $g \in \SL{n}{q}$, $\theta(\det(g))\unipotentCharacter{\mu}(g) = \unipotentCharacter{\mu}(g)$. Therefore ${\restricted{(\theta(\det)\unipotentCharacter{\mu})}{\SL{n}{q}} = \restricted{\unipotentCharacter{\mu}}{\SL{n}{q}}}$.
		
		Let $C$ the cyclic group of diagonal matrices with $1$ on the diagonal, except possibly in the top left cell. We have ${\cardinality{C} = q-1}$, ${C \cap \SL{n}{q} = 1}$, ${C \ \SL{n}{q} = \GL{n}{q}}$ and $C$ is a cyclic subgroup of $\GL{n}{q}$.
		
		It is easy to see that the characters $\theta(\det)\unipotentCharacter{\mu}$ are all distinct by evaluating them on the elements of $C$ (there exist characters $\chi$ of $\GL{n}{q}$ for which $\theta(\det)\chi$ are not all distinct). As such, the restrictions of the $q-1$ distinct irreducible characters $\theta(\det)\unipotentCharacter{\mu}$ are equal to $\restricted{\unipotentCharacter{\mu}}{\SL{n}{q}}$. Therefore, by \autoref{theorem:restriction_to_normal_subgroup}, $\restricted{\unipotentCharacter{\mu}}{\SL{n}{q}}$ splits into a single irreducible character, namely, it is irreducible itself.
	\end{proof}

	By identifying $\unipotentCharacter{\mu} \define \restricted{\unipotentCharacter{\mu}}{\SL{n}{q}}$, using \autoref{corollary:restriction_of_unipotent_characters_to_sl}, we can deduce that the non-projective version of \autoref{theorem:nonvanishing_characters_of_gln} implies the non-projective version of \autoref{theorem:nonvanishing_characters_of_sln}.
	
	\subsection{The Projective Groups} \label{section:projective_unipotent_restriction}
	\paragraph{ } In this section, we assume the non-projective versions of the main theorems, and deduce the projective counterparts.
	\begin{lemma}
		Let $G$ and $H$ be finite groups, with $p \colon G \twoheadrightarrow H$. Let ${\chi \in \irreducibleCharacters{G}}$, and let $(\pi, V)$ be its corresponding representation. Define 
		\[\ker(\chi) \define \ker(\pi) = \{g \in G \colon \chi(g) = \chi(1)\}.\]
		Assume that $\ker(\chi) \supergroup \ker(p)$. Then $\chi'(h) \define \chi(p^{-1}(h))$ is well defined, and is an irreducible character of $H$.
	\end{lemma}
	\begin{proof}
		Since $\ker(\pi) \supergroup \ker(p)$, there is a well defined representation $\pi'$ of $H$ with ${\pi(g) = \pi'(p(g))}$ for every $g \in G$. It is irreducible since $\Ima(\pi') = \Ima(\pi)$, which stabilizes no nontrivial subspace of $V$.
		
		By taking the trace of the representations, with the trace of $\pi'$ equal to $\chi'$, the proposition follows.
	\end{proof}
	
	\begin{corollary} \label{corollary:projection_of_nonvanishing_characters}
		Let $G' \subgroup G$ be groups, and let $\specialCharacterSet{G} \subseteq \irreducibleCharacters{G}$ be a set \textCMarksSubgroupsThatDoNotContainH{that}{$G$}{$G'$}.	Let $p \colon G \twoheadrightarrow H$, and assume that ${\forall \chi \in \specialCharacterSet{G} \colon \ker(\chi) \supergroup \ker(p)}$. Define $H' \define p(G')$ and
		\[\specialCharacterSet{H} \define \{ g \mapsto \chi(p^{-1}(g)): \chi \in \specialCharacterSet{G}\} \subseteq \irreducibleCharacters{H}.\]
		Then, \textCMarksSubgroupsThatDoNotContainH{$\specialCharacterSet{H}$}{$H$}{$H'$}.
	\end{corollary}
	\begin{proof}
		Let $J \subgroup H$, such that $J \not \supergroup H'$, and let $K \define p^{-1}(J)$.
		
		Since $K$ does not contain $G'$, there exists some $\chi \in \specialCharacterSet{G}$ such that ${\chiKContainsTrivial{\chi}{K}}$. Denote by $\chi' \in \specialCharacterSet{H}$ the character defined by $\chi'(g) = \chi(p^{-1}(g))$. We have
		\[\characterSum{\chi'}{J} = \frac{1}{\cardinality{J}}\sum_{h \in J}\chi'(h) = \frac{1}{\cardinality{J}} \frac{\cardinality{J}}{\cardinality{K}} \sum_{g \in K}\chi(g) = \characterSum{\chi}{K} > 0\]
		which implies that $\chiKContainsTrivial{\chi'}{J}$.
	\end{proof}

	\autoref{lemma:induced-trivial-character-split} implies that the kernels of the unipotent characters contain the center of $\GL{n}{q}$, as they are spanned by the characters $\trivialCharacter{\lambda}$, that do satisfy this. Thus, we may identify $\unipotentCharacter{\mu}$ with the corresponding characters as in \autoref{corollary:projection_of_nonvanishing_characters}.
	
	Using \autoref{corollary:projection_of_nonvanishing_characters}, we deduce that the non-projective versions of \autoref{theorem:nonvanishing_characters_of_gln} and \autoref{theorem:nonvanishing_characters_of_sln} imply their projective counterparts.

	\subsection{$\GL{n}{q}$} \label{section:nonvanishing_characters_ge_3}
	\paragraph{ } To prove the main theorems for $\GL{n}{q}$, $n \ge 3$, we need to know the possibilities for subgroups of $\GL{n}{q}$ that act $2$-transitively on the projective space. The following theorem summarizes these possibilities:
	
	\begin{theorem}[\cite{2_transitive_collineation_groups}] \label{theorem:2_transitive_collineation_groups}
		If $K \subgroup \GammaL{n}{q}, n \ge 3$ and $K$ is $2$-transitive on $\projective^{n-1}\field{q}$, then either $K \supergroup \SL{n}{q}$ or $K = \alternating{7}$, with $n = 4$ and $q = 2$.
	\end{theorem}

	This theorem, \autoref{corollary:no-non-vanishing-subgroups-3} and the results of \autoref{section:projective_unipotent_restriction} show that \textCMarksSubgroupsThatAreNotxkTransitive{$\twoTransitivityCharacter$}{$\PGL{3}{q}$}{\projectiveVector{e}}{3}, since \xkGTransitivity{\projectiveVector{e}}{3}{$\PGL{3}{q}$} implies $2$-transitivity.
	
	Thus, using \autoref{corollary:glnq-stabilizer-character}, we see that $\forall n \ge 3$, \textCMarksSubgroups{$\twoTransitivityCharacter \circ \trivialCharacter{(n-3)}$}{$\GL{n}{q}$} that is not \xkGTransitive{\projectiveVector{e}}{2}{$\GL{n}{q}$}. Since $\GL{n}{q}$ is $2$-transitive with its projective action, we deduce that \textCMarksSubgroupsThatAreNotkTransitive{$(\transitivityCharacter + \twoTransitivityCharacter) \circ \trivialCharacter{(n-3)}$}{$\GL{n}{q}$}{2} with the restricted action on $\projective^{n-1}\field{q}$.
	
	Outside the exceptional case of the theorem, the subgroups of $\GL{n}{q}$ that do not act $2$-transitively on $\projective^{n-1}\field{q}$ are precisely those that do not contain $\SL{n}{q}$, by \autoref{theorem:2_transitive_collineation_groups}.
	
	Thus, to find the irreducible characters mentioned in the main theorems (for all but $\GL{4}{2}$), it remains to decompose $\twoTransitivityCharacter \circ \trivialCharacter{(n-3)}$ to irreducible characters.
	
	By the definition of $\twoTransitivityCharacter$,
	\[\twoTransitivityCharacter \circ \trivialCharacter{(n-3)} = (\trivialCharacter{(1^3)} - \trivialCharacter{(3)}) \circ \trivialCharacter{(n-3)} = \trivialCharacter{(n-3, 1^3)} - \trivialCharacter{(n-3, 3)}.\]
	
	These can be decomposed using \autoref{lemma:induced-trivial-character-split}, to get:
	\[\twoTransitivityCharacter \circ \trivialCharacter{(n-3)} =
	\begin{cases}
	\unipotentCharacter{(1^3)} + 2\unipotentCharacter{(2, 1)} 	& n = 3 \\
	\unipotentCharacter{(1^4)} + 3\unipotentCharacter{(2, 1^2)} + 2\unipotentCharacter{(2^2)} + 2\unipotentCharacter{(3, 1)}	& n = 4 \\
	\unipotentCharacter{(n-3, 1^3)} + 2\unipotentCharacter{(n-3, 2, 1)} + 3\unipotentCharacter{(n-2, 1^2)} + 2\unipotentCharacter{(n-2, 2)} + 2\unipotentCharacter{(n-1, 1)}	& n \ge 5
	\end{cases}
	\]
	This concludes the proof of \autoref{theorem:nonvanishing_characters_of_gln}.

	\subsection{$\GL{4}{2}$} \label{section:gl42}
	\paragraph{ } In this section, we study the exceptional case of $\GL{4}{2}$. Since the underlying field is $\field{2}$, we have $\GL{4}{2} = \SL{4}{2} = \PGL{4}{2} = \PSL{4}{2}$.
	
	The problem with this case is due to \autoref{theorem:2_transitive_collineation_groups}. Note that \textCMarksProperSubgroups{${\twoTransitivityCharacter\circ \trivialCharacter{(1)}}$}{$\GL{4}{2}$} that is not the $\alternating{7} \subgroup \GL{4}{2}$ referenced in the theorem.
	
	We wish to find a character $\psi \in \irreducibleCharacters{\GL{4}{2}}$ such that \textChiDoesntVanishOnK{\psi}{\alternating{7}}. For this goal, we use the following proposition to better understand the inclusion of $\alternating{7}$ in $\GL{4}{2}$.
	
	\begin{proposition}[e.g. {\cite[Chapter 7.7]{permutation_groups}}]
		$\GL{4}{2} \cong \alternating{8}$. Under this isomorphism, the $2$-transitive $\alternating{7}$ is conjugate to the stabilizer of an element in the natural action of $\alternating{8}$ on 8 elements.
	\end{proposition}
	
	Thus, we use the representation theory of $\alternating{8}$. For that, we consider the representation theory for $\symmetric{n}$ (see, for instance, \cite{sn_characters}), and use \autoref{theorem:restriction_to_normal_subgroup}. This leads to the following proposition.
	
	\begin{proposition} \label{proposition:a8_character_degrees}
		The characters of $\alternating{8}$ have degrees (with multiplicities):
		
		$\{1, 7, 14, 20, 21, 21, 21, 28, 35, 45, 45, 56, 64, 70\}$.
	\end{proposition}

	\paragraph{ } To prove the proposition, we use the following. See \cite{sn_characters} for information about standard tableaux.

	\begin{proposition}[e.g. {\cite[Chapters 7.2 and 7.3]{sn_characters}}]
		For every $n \in \positiveIntegers$, ${\irreducibleCharacters{\symmetric{n}} = \{ \snCharacter{\lambda} \colon \lambda \text{ is a partition of } n \}}$, for some characters with $\dim(\snCharacter{\lambda}) = f^\lambda = $ the number of standard tableaux of shape $\lambda$.
		
		In addition, $\sgn \snCharacter{\lambda} = \snCharacter{\lambda'}$, with $\sgn$ being the sign representation.
	\end{proposition}
	
	\begin{proof}[Proof of \autoref{proposition:a8_character_degrees}]
		By substituting $n = 8$ and calculating the characters, the characters of $\symmetric{n}$ have degrees (with multiplicities) 
		\[\{1, 1, 7, 7, 14, 14, 20, 20, 21, 21, 28, 28, 35, 35, 42, 56, 56, 64, 64, 70, 70, 90\},\] 
		with every pair belonging to conjugate partitions.
		
		By \autoref{theorem:restriction_to_normal_subgroup}, noting that 	$\restricted{\snCharacter{\lambda'}}{\alternating{n}} = \restricted{(\sgn\snCharacter{\lambda})}{\alternating{n}} = \restricted{\snCharacter{\lambda}}{\alternating{n}}$, we deduce that the characters of $\alternating{8}$ are precisely:
		\begin{enumerate}
			\item The restrictions of $\snCharacter{\lambda}$ for $\lambda \ne \lambda'$
			\item Two summands of $\restricted{\snCharacter{\lambda}}{\alternating{n}}$ of equal dimensions for $\lambda = \lambda'$
		\end{enumerate}
		
		Thus, we obtain the following degrees:
		\[\{1, 7, 14, 20, 21, 21, 21, 28, 35, 45, 45, 56, 64, 70\}.\]
		These are all of the characters of $\alternating{8}$, as the sum of their squares is $\cardinality{\alternating{8}}$.
	\end{proof}

	The last thing we need to do to complete the proof of \autoref{theorem:nonvanishing_characters_of_gl42} is understand the decomposition of $\induced{\trivialCharacter{\alternating{7}}}{\alternating{7}}{\GL{4}{2}}$.
	
	\begin{proposition}
		$\induced{\trivialCharacter{\alternating{7}}}{\alternating{7}}{\GL{4}{2}} = \trivialCharacter{\GL{4}{2}} + \psi$, where $\psi \in \irreducibleCharacters{\GL{4}{2}}$ is the single irreducible of dimension $7$.
	\end{proposition}
	\begin{proof}
		Let $\theta = \induced{\trivialCharacter{\alternating{7}}}{\alternating{7}}{\GL{4}{2}}$. Observe that $\dim(\theta) = \subgroupIndex{\alternating{8}}{\alternating{7}} = 8$.	By \autoref{proposition:stabilizer-induced-character}, $\innerProduct{\theta}{\trivialCharacter{\GL{4}{2}}} = 1$. Thus, since $\GL{4}{2} \cong \alternating{8}$ has a unique nontrivial character of degree at most $7$, the claim follows.
	\end{proof}

	\paragraph{Remark} By calculating the degrees of the characters, we deduce that the character $\psi$ is not a unipotent character.
	
	We can now state the following proposition.
	
	\begin{proposition}
		The only nontrivial $\psi \in \irreducibleCharacters{\GL{4}{2}}$ such that \textChiDoesntVanishOnK{\psi}{\alternating{7}} is the character of $\GL{4}{2}$ of degree $7$.
	\end{proposition}
	\begin{proof}
		Let  $\psi \in \irreducibleCharacters{\GL{4}{2}}$ be the single irreducible character of dimension $7$.
		By \nameref{frobenius_reciprocity},
		\[
		\characterSum{\theta}{\alternating{7}} = \innerProduct{\theta}{\induced{\trivialCharacter{\alternating{7}}}{\alternating{7}}{\GL{4}{2}}}_{\GL{4}{2}} =
		\innerProduct{\theta}{\trivialCharacter{\GL{4}{2}} + \psi}_{\GL{4}{2}}.
		\]
		So the only characters where the sum is non-zero are the trivial character and $\psi$.
	\end{proof}
	
	This completes the proof of \autoref{theorem:nonvanishing_characters_of_gl42}

	\section{The Main Theorems for $n = 2$}  \label{section:nonvanishing_characters_2}
	\paragraph{ } In this section, we prove the main theorems for $n = 2$. We find a character $\additionalForGLTwo$ such that for every $K \in \classicalFamilySymbol{3} \cup \classicalFamilySymbol{6} \cup \classicalExceptions$, either \textChiDoesntVanishOnK{\additionalForGLTwo}{K} or \textChiDoesntVanishOnK{\transitivityCharacter}{K}. Then, by \autoref{corollary:no-non-vanishing-subgroups-2}, the non-projective case of \autoref{theorem:nonvanishing_characters_of_gl2} follows.
	
	\paragraph{ } The characters in the $n=2$ case are slightly simpler than in the general case. We use the character tables in \cite{gl2_characters}, and refer the readers to them for definitions of the characters $\pi(\chi)$ and $\omega_0^+$. The two notations are related as follows, with $\alpha, \beta \in \irreducibleCharacters{\fieldInvertible{q}}$ and $\chi \in \irreducibleCharacters{\fieldInvertible{q^2}}$:
	\begin{enumerate}
		\item $\rho(\mu(\alpha, \beta)) = \alpha \circ \beta$
		\item $\rho'(\alpha)$ is the linear character $\rho'(\alpha)(g) = \alpha(\det(g))$
		\item $\bar{\rho}(\alpha) = \alpha \circ \alpha-\rho'(\alpha)$. In particular, $\bar{\rho}(\trivialCharacter{\fieldInvertible{q}}) = \unipotentCharacter{(1, 1)} = \transitivityCharacter$.
		\item $\pi(\chi)$ are additional characters of $\GL{2}{q}$ called the cuspidal characters.
	\end{enumerate}
	
	\begin{definition} \label{definition:additional_character_2}
		Define $\defineNotation{\additionalForGLTwo}$ according to the following cases:
		\begin{enumerate}
			\item If $q=3$ and $G$ is $\SL{2}{q}$ or $\PSL{2}{q}$, $\additionalForGLTwo$ is $\omega_0^+$, as defined in the character table of these groups.
			\item Otherwise, if $q$ is odd, $\additionalForGLTwo$ is $\defineNotation{\cuspidalForGLTwo} \define \pi(\phi)$, for $\phi \in \irreducibleCharacters{\fieldInvertible{q^2}}$, of order\declareFootnote{phi_order}{The order condition makes sure the character projects to the projective group} $q+1$ in the character group.
			\item If $q$ is even and $q > 2$, $\additionalForGLTwo$ is $\alpha \circ \bar{\alpha}$, where $\alpha \ne \trivialCharacter{\fieldInvertible{q}}$.
			\item If $q = 2$, $\additionalForGLTwo$ is $\defineNotation{\cuspidalForGLTwo} \define \pi(\phi)$, for any of the nontrivial $\phi \in \irreducibleCharacters{\fieldInvertible{4}}$.
		\end{enumerate}
	\end{definition}

	By following the character tables, we see that the characters $\cuspidalForGLTwo$ and $\alpha \circ \bar{\alpha}$ do restrict irreducibly to the special groups, for $q \ne 3$. In the case of $q=3$, $\cuspidalForGLTwo$ splits to two summands -- one of which is the character $\omega_0^+$. In addition, the characters $\omega_0^+$, $\cuspidalForGLTwo$ and $\alpha \circ \bar{\alpha}$ project to the projective groups. Therefore, by claims similar to those made in \autoref{section:special_linear_unipotent_restriction} and \autoref{section:projective_unipotent_restriction}, the theorems for the projective and the special groups follow from the case of $\GL{2}{q}$.
	
	\subsection{$\classicalFamily{3}{\GL{2}{q}}$}
	\paragraph{ } The family $\classicalFamily{3}{\GL{2}{q}}$ is the family of stabilizers of the structure of $\field{q}^2$ as a $1$-dimensional vector space over $\field{q^2}$, that is, elements that act on $\field{q}^2 = \field{q^2}$ as either multiplication by invertible elements of $\field{q^2}$, or an automorphism of $\field{q^2}$ over $\field{q}$.
	We use \cite{gl2_characters} for the various character tables.

	\paragraph{$q$ is odd} First, we handle the case when $q$ is odd.
	
	Since $\fieldInvertible{q}$ is cyclic of even order, there exists some non square $\Delta \in \fieldInvertible{q}$. Therefore the group $\{x^2: x \in \fieldInvertible{q}\}$ is a proper subgroup of $\fieldInvertible{q}$. This implies that $x^2-\Delta$ is irreducible over $\field{q}$.
	
	Thus, we can choose the base $\{1, \delta\}$ for the field extension, where $\delta ^ 2 = \Delta$. Since every $K \in \classicalFamily{3}{\GL{2}{q}}$ respects the structure of $\field{q^2}$, we deduce:
	
	\begin{proposition}
		If $q$ is odd and $K \in \classicalFamily{3}{\GL{2}{q}}$, $K$ is conjugate to
		
		\[
		\left\{
		\begin{bmatrix}
		a		& \Delta b		 \\
		b		& a	 \\
		\end{bmatrix}
		,
		\begin{bmatrix}
		a		& -\Delta b		 \\
		b		& -a	 \\
		\end{bmatrix}
		\right\} \subgroup \GL{2}{q}.
		\]
	\end{proposition}

	We can conclude (since $\groupCenter{\GL{2}{q}} \le K$) that $K$ has the following conjugacy classes (where $a$ and $b$ range over $\field{q}$ such that the matrices are invertible):
	\[		\left\{
	\begin{bmatrix}
	a		& \Delta b		 \\
	b		& a	 \\
	\end{bmatrix}
	,
	\begin{bmatrix}
	0		& -\Delta b		 \\
	b		& 0	 \\
	\end{bmatrix} c \text{ times},
	\begin{bmatrix}
	a		& 0		 \\
	0		& -a	 \\
	\end{bmatrix} (q+1-c) \text{ times}
	\right\}.
	\]

	Using the character tables, we can see that:
	
	\begin{proposition} \label{proposition:c_3_gl2_odd_q}
		If $q$ is odd and $K \in \classicalFamily{3}{\GL{2}{q}}$, \textChiDoesntVanishOnK{\cuspidalForGLTwo}{K}.
	\end{proposition}

	\begin{proof}
		\begin{multline*}
			\begin{aligned}
				\cardinality{K}\characterSum{\cuspidalForGLTwo}{K} & = (q-1)\sum_{x \in \fieldInvertible{q}}\phi(x)-\sum_{z \in \fieldInvertible{q^2} \setminus \fieldInvertible{q}}(\phi(z) + \phi(\bar{z})) \\
				&-c \sum_{x \in \fieldInvertible{q}}(\phi(\delta x) + \phi(-\delta x)) \\
				& = (q+1)\sum_{x \in \fieldInvertible{q}}\phi(x)-2\sum_{z \in \fieldInvertible{q^2}}\phi(z)-2c \phi(\delta) \sum_{x \in \fieldInvertible{q}}\phi(x) \\
				& = (q+1+2c)(q-1)
			\end{aligned}
		\end{multline*}
		Since $0 \le c \le q+1$, $\cardinality{K} = 2(q^2-1)$, and $\characterSum{\cuspidalForGLTwo}{K}$ is an integer, we deduce that $c = \frac{q+1}{2}$. Thus, $\characterSum{\cuspidalForGLTwo}{K} = 1$.
	\end{proof}

	This shows the desired result for the family $\classicalFamily{3}{\GL{2}{q}}$, for odd $q \ne 3$.
	
	\paragraph{The case $q = 3$}
	While \autoref{proposition:c_3_gl2_odd_q} shows the desired result for $\GL{2}{3}$, and projects to $\PGL{2}{3}$, the special cases are slightly more complicated. In these cases, as can be seen in the character tables, $\cuspidalForGLTwo$ (both in the projective and non-projective cases) decomposes to two $1$-dimensional characters called the oscillator characters ($\omega^{\pm}_o$ in \cite{gl2_characters}). Either of these satisfies the desired result, and we arbitrarily choose $\omega^{+}_o$.
	
	\paragraph{$q$ is even} Here, we handle the case when $q$ is even.
	
	If $q = q_0^2$, consider the trace morphism $\trace{\field{q}}{\field{q_0}}(x) = x^2 + x$, with its image $\field{q_0}$. Choosing $\Delta \in \field{q} \setminus \field{q_0}$, the polynomial $x^2 + x-\Delta$ is irreducible over $\field{q}$.
	
	If $q \ne q_0^2$, $\field{q}$ is not an extension of $\field{4}$. As such, the polynomial $x^2 + x-1$ is irreducible over $\field{q}$. Choose $\Delta = 1$.
	
	Thus, we can choose the basis $\{1, \delta\}$ for the field extension $\field{q^2}/\field{q}$, where $\delta ^ 2 + \delta = \Delta$. Since every $K \in \classicalFamily{3}{\GL{2}{q}}$ respects the structure of $\field{q^2}$, we deduce:
	
	\begin{proposition}
		If $q$ is even and $K \in \classicalFamily{3}{\GL{2}{q}}$, $K$ is conjugate to
		
		\[
		\left\{
		\begin{bmatrix}
		a		& \Delta b		 \\
		b		& a + b	 \\
		\end{bmatrix}
		,
		\begin{bmatrix}
		a		& a + \Delta b		 \\
		b		& a	 \\
		\end{bmatrix}
		\right\} \subgroup \GL{2}{q}.
		\]
	\end{proposition}

	Thus, we can conclude (since $\groupCenter{\GL{2}{q}} \le K$) that $K$ has the following conjugacy classes (where $a$, $b$ and $x$ range over $\field{q}$ such that the matrices are invertible):
	\[		\left\{
	\begin{bmatrix}
	a		& \Delta b		 \\
	b		& a + b	 \\
	\end{bmatrix}
	,
	\begin{bmatrix}
	x		& 1		 \\
	0		& x	 \\
	\end{bmatrix} q+1 \text{ times}
	\right\}
	\]
	
	Using the character tables, we can see that:
	
	\begin{proposition}
		If $q \ne 2$ is even and $K \in \classicalFamily{3}{\GL{2}{q}}$, \textChiDoesntVanishOnK{\alpha \circ \bar{\alpha}}{K}.
	\end{proposition}
	
	\begin{proof}
		\begin{multline*}
			\begin{aligned}
				\cardinality{K}\characterSum{\alpha \circ \bar{\alpha}}{K} = 2(q+1)\sum_{x \in \fieldInvertible{q}}1 = 2(q^2-1).
			\end{aligned}
		\end{multline*}
		Dividing by $\cardinality{K} = 2(q^2-1)$, we get $\characterSum{\alpha \circ \bar{\alpha}}{K} = 1$.
	\end{proof}
	
	By the character tables, if $\alpha \ne \trivialCharacter{\fieldInvertible{q}}$, $\alpha \circ \bar{\alpha}$ is irreducible, and does restrict irreducibly to the special and projective cases. Such a choice exists for all even $q \ne 2$.	This shows the desired result for the family $\classicalFamily{3}{\GL{2}{q}}$, for even $q \ne 2$.
	
	\paragraph{The case $q = 2$}
	In this case, a simple calculation shows that $\additionalForGLTwo$ marks the single proper subgroup in $\classicalFamily{3}{\GL{2}{2}}$ which is of order $3$.

	\subsection{$\classicalFamily{6}{\GL{2}{q}}$ and $\classicalExceptions(\GL{2}{q})$} \label{section:c6_and_s_for_n_eq_2}

	\begin{theorem}[{\cite[Table 8.3]{maximal_subgroups}}] \label{theorem:c6_and_classical_exceptions_of_sl2}
		$(\classicalFamilySymbol{6} \cup \classicalExceptions)(\SL{2}{q})$ are given in \autoref{table:c6_and_classical_exceptions_of_sl2}.
		\declareAbcCounter{c6ClassicalExceptionsCounter}
		\begin{table}[]
			\centering
			\caption{$(\classicalFamilySymbol{6} \cup \classicalExceptions)(\SL{2}{q})$} \label{table:c6_and_classical_exceptions_of_sl2}
			\begin{threeparttable}
				\begin{tabular}{cCCC}
					\hline
					Case & K' \le \SL{2}{q} & \text{Conditions on } q
					& \text{Order}\\ \hline \hline

					\abcCounter{c6ClassicalExceptionsCounter} & 2 \symmetric{4}
					&\makecell{q\text{ is prime} \\ q \equiv \pm 1 \mod 8}
					& 48  \\ \hline
					
					\abcCounter{c6ClassicalExceptionsCounter} & 2\alternating{4}
					&\makecell{q\text{ is prime} \\ q \equiv \pm 3, 5, \pm 11, \pm 13, \pm 19 \mod 40}
					&24 \\ \hline
					
					\abcCounter{c6ClassicalExceptionsCounter} \label{table:c6_and_classical_exceptions_of_sl2_alternating5_1} & 2\alternating{5}
					&\makecell{q\text{ is prime } \\ q \equiv \pm 1 \mod 10 }
					& 120  \\ \hline

					\abcCounter{c6ClassicalExceptionsCounter} \label{table:c6_and_classical_exceptions_of_sl2_alternating5_2} & 2\alternating{5}
					&\makecell{q\text{ is the square of a prime }p \\ p \equiv \pm 3 \mod 10 }
					& 120  \\ \hline
				\end{tabular}
			\end{threeparttable}
		\end{table}
	\end{theorem}

	We can see that the orders of these families are bounded. If ${K \in (\classicalFamilySymbol{6} \cup \classicalExceptions)(\GL{2}{q})}$ does not act transitively on $\projective^{1}\field{q}$, then \textChiDoesntVanishOnK{\transitivityCharacter}{K}, and our main claim follows.
	
	As such, we use an argument about the transitivity of their action on $\projective^1\field{q}$ to bound $q$, namely:
	
	\begin{proposition} \label{proposition:c6_and_classical_exceptions_of_sl2_order_bound}
		Assume $K \in (\classicalFamilySymbol{6} \cup \classicalExceptions)(\GL{2}{q})$, and that it acts transitively on $\projective^1 \field{q}$. Denote $K' \define K \cap \SL{2}{q}$. Then, $(q+1) \divides \cardinality{K'}$.
	\end{proposition}
	
	\begin{proof}
		By \autoref{table:c6_and_classical_exceptions_of_sl2}, we can see that $q$ is odd. Let $J \define \factor{K}{\groupCenter{\GL{2}{q}}}$ and $J' \define \factor{K'}{\groupCenter{\SL{2}{q}}}$.
		
		By \eqref{corollary:orbit_stabilizer_orders}, $(q+1) \divides \cardinality{J}$. As $\cardinality{K'}$ is given by \autoref{table:c6_and_classical_exceptions_of_sl2}, by the maximality of $K'$ along with its order\declareFootnote{maximality_and_order}{Otherwise, $K' \groupCenter{\SL{2}{q}} = \SL{2}{q}$. However, this implies that $2 \cardinality{K'} = \cardinality{\SL{2}{q}}$ (since $q$ is odd, $\cardinality{\groupCenter{\SL{2}{q}}} = 2$). None of the orders from \autoref{table:c6_and_classical_exceptions_of_sl2} satisfies this condition.}, $K' \supergroup \groupCenter{\SL{2}{q}}$. Thus, $\cardinality{K'} = 2\cardinality{J'}$. Then, lifting $J'$ to $J \le \PGL{2}{q}$, we get either $\cardinality{J} = \cardinality{J'}$ or $\cardinality{J} = 2\cardinality{J'}$. In either case, $\cardinality{J}$ divides $\cardinality{K'}$. So we get that $(q+1) \divides \cardinality{K'}$.
	\end{proof}
	
	To ease the handling of cases \ref{table:c6_and_classical_exceptions_of_sl2_alternating5_1} and \ref{table:c6_and_classical_exceptions_of_sl2_alternating5_2}, we use the following:
	
	\begin{proposition}[{\cite[Sections 258-259]{psl2_subgroups}}]
		$\alternating{5} \le \PSL{2}{q}$ is, up to conjugation, generated by the following: 
		
		If $q \equiv \pm 1 \mod 5$
		\[
		V_5 \define 
		\begin{bmatrix}
		\zeta		& 0		 \\
		0		& \zeta^{-1}	 \\
		\end{bmatrix}
		,
		\text{ some } V_2 \text{ of order } 2
		\]
		where $\zeta$ is a primitive root of unity of order $5$.

		If $q \equiv 0 \mod 5$
		\[
		V_5 \define
		\begin{bmatrix}
		1		& \mu		 \\
		0		& 1	 \\
		\end{bmatrix}
		,
		\text{ some } V_2 \text{ of order } 2
		\]
		where $\mu \ne 0$.
	\end{proposition}
	
	\begin{corollary} \label{corollary:c6_and_classical_exceptions_of_sl2_alternating_5_q}
		Assume $K \in (\classicalFamilySymbol{6} \cup \classicalExceptions)(\GL{2}{q})$, $K$ acts transitively on $\projective^1 \field{q}$ and one of cases \ref{table:c6_and_classical_exceptions_of_sl2_alternating5_1} or \ref{table:c6_and_classical_exceptions_of_sl2_alternating5_2} holds. Then $(q+1) \divides 24$.
	\end{corollary}
	\begin{proof}
		In these cases, the projective image of $K = 2\alternating{5}$ (which is the lifting of $\alternating{5} \subgroup \PSL{2}{q}$ to $\PGL{2}{q}$) contains a subgroup of order $5$ stabilizing $\projectiveVector{e}_1$. As such, $5 \divides \cardinality{\stab{K}{\projectiveVector{e}_1}}$
		
		Since $K$ acts transitively on $\projective^1 \field{q}$, by \eqref{corollary:orbit_stabilizer_orders}, $(q+1) \cardinality{\stab{K}{\projectiveVector{e}_1}} = \cardinality{K}$. Therefore, $5(q+1) \divides \cardinality{2\alternating{5}} = 120$. This implies that $(q+1) \divides 24$.
	\end{proof}
	
	By using these \autoref{proposition:c6_and_classical_exceptions_of_sl2_order_bound} and \autoref{corollary:c6_and_classical_exceptions_of_sl2_alternating_5_q}, along with the conditions in \autoref{table:c6_and_classical_exceptions_of_sl2}, we arrive at the following:
	
	\begin{proposition} \label{proposition:nontransitive_sl2_c6_exceptional}
		Only the following possibilities for $K \cap \SL{2}{q}$ can act transitively on $\projective^{1}\field{q}$, for $K \in (\classicalFamilySymbol{6} \cup \classicalExceptions)(\GL{2}{q})$:
		\begin{multicols}{3}
			\begin{enumerate}
				\item $2\alternating{4} \le \SL{2}{5}$
				\item $2\symmetric{4} \le \SL{2}{7}$
				\item $2\alternating{5} \le \SL{2}{11}$
				\item $2\symmetric{4} \le \SL{2}{23}$
				\item $2\symmetric{4} \le \SL{2}{31}$
				\item $2\symmetric{4} \le \SL{2}{47}$
			\end{enumerate}
		\end{multicols}
	
		We do not claim these possibilities for $K$ do act transitively on the projective space.
	\end{proposition}
	
	These can be checked by hand, using the character table in \cite{gl2_characters}, to arrive at:

	\begin{proposition}
		For the possibilities listed in \autoref{proposition:nontransitive_sl2_c6_exceptional}, the different $K \in (\classicalFamilySymbol{6} \cup \classicalExceptions)(\GL{2}{q})$, along with a set of generators (up to conjugation) and a character $\chi$ such that \textChiDoesntVanishOnK{\chi}{K}, are given in \autoref{table:nontransitive_sl2_c6_exceptional_with_generators_and_characters}.
		\begin{table}[]
		\centering
			\caption{Generators and Characters} \label{table:nontransitive_sl2_c6_exceptional_with_generators_and_characters}
			\begin{threeparttable}
				\begin{tabular}{CCCC}
					\hline
					q & K \cap \SL{2}{q} & \text{Generators in } \PGL{2}{q}
					& \text{Character}\\ \hline \hline
					5 & 2 \alternating{4} 
					& \begin{bmatrix}3&4\\2&3\end{bmatrix},\begin{bmatrix}1&0\\1&2\end{bmatrix} 
					& \cuspidalForGLTwo \\ \hline
					7 & 2 \symmetric{4} 
					& \begin{bmatrix}2&6\\5&5\end{bmatrix},\begin{bmatrix}4&6\\3&5\end{bmatrix} 
					& \cuspidalForGLTwo \\ \hline
					7 & 2 \symmetric{4} 
					& \begin{bmatrix}2&0\\0&1\end{bmatrix},\begin{bmatrix}6&5\\6&4\end{bmatrix} 
					& \cuspidalForGLTwo \\ \hline
					11 & 2 \alternating{5} 
					& \begin{bmatrix}8&7\\6&4\end{bmatrix},\begin{bmatrix}8&8\\8&4\end{bmatrix} 
					& \cuspidalForGLTwo \\ \hline
					11 & 2 \alternating{5} 
					& \begin{bmatrix}3&9\\5&8\end{bmatrix},\begin{bmatrix}6&2\\6&4\end{bmatrix} 
					& \cuspidalForGLTwo \\ \hline
					23 & 2 \symmetric{4} 
					& \begin{bmatrix}8&2\\18&18\end{bmatrix},\begin{bmatrix}15&1\\1&8\end{bmatrix} 
					& \cuspidalForGLTwo \\ \hline
					23 & 2 \symmetric{4} 
					& \begin{bmatrix}1&15\\17&8\end{bmatrix},\begin{bmatrix}2&11\\17&17\end{bmatrix} 
					& \cuspidalForGLTwo \\ \hline
					31 & 2 \symmetric{4} 
					& \begin{bmatrix}11&19\\7&4\end{bmatrix},\begin{bmatrix}1&13\\30&26\end{bmatrix} 
					& \transitivityCharacter \\ \hline
					31 & 2 \symmetric{4} 
					& \begin{bmatrix}5&25\\18&6\end{bmatrix},\begin{bmatrix}0&22\\20&0\end{bmatrix} 
					& \transitivityCharacter \\ \hline
					47 & 2 \symmetric{4} 
					& \begin{bmatrix}34&10\\3&13\end{bmatrix},\begin{bmatrix}24&32\\11&6\end{bmatrix} 
					& \transitivityCharacter \\ \hline
					47 & 2 \symmetric{4} 
					& \begin{bmatrix}11&2\\29&36\end{bmatrix},\begin{bmatrix}46&44\\11&21\end{bmatrix} 
					& \transitivityCharacter \\ \hline
				\end{tabular}
			\end{threeparttable}
		\end{table}
	\end{proposition}

	Thus, either $\chiContainsTrivial{\cuspidalForGLTwo}{K}$ or $\chiContainsTrivial{\transitivityCharacter}{K}$ every $K \in (\classicalFamilySymbol{6} \cup \classicalExceptions)(\GL{2}{q})$. In addition, both characters restrict to the projective and special cases.
	
	This completes the proof of \autoref{theorem:nonvanishing_characters_of_gl2} and \autoref{theorem:nonvanishing_characters_of_sl2}.
	
	\section{Open Questions} \label{section:open_questions}
	
	\paragraph{Minimality of $\specialCharacterSetGLnq{n}{q}$}
	In \autoref{theorem:nonvanishing_characters_of_gln} we have shown that the set $\specialCharacterSetGLnq{n}{q} \subseteq \irreducibleCharacters{\GL{n}{q}}$ defined in \autoref{section:main_results} suffices to mark every subset of $\GL{n}{q}$ that does not contain $\SL{n}{q}$. We have not shown that this set is indeed minimal. This raises the following question, and similar questions for the relatives of $\GL{n}{q}$.
	
	\begin{question}
		Is there a set $\specialCharacterSetGLnq{n}{q} \subseteq \irreducibleCharacters{\GL{n}{q}}$ with at most $4$ nontrivial irreducible characters \textCMarksSubgroupsThatDoNotContainH{that}{$\GL{n}{q}$}{$\SL{n}{q}$}?
	\end{question}
	
	\paragraph{$\specialCharacterSet{G}$ for other groups} Given a group $G$, we denote by $\specialMinCharacterSetCardinality{G}$ the minimal cardinality of a subset of nontrivial irreducible characters \textCMarksProperSubgroups{that}{$G$}.
	
	\begin{question}
		Given a finite group $G$, what is $\specialMinCharacterSetCardinality{G}$?
	\end{question}

	\begin{question}
		Given a finite group $G$, what subsets of characters achieve $\specialMinCharacterSetCardinality{G}$? Is there a unique minimal set?
	\end{question}

	\paragraph{Families of groups} Given a family of finite groups, we consider $\max_G(\specialMinCharacterSetCardinality{G})$, where the maximum ranges over all of the groups in the family.

	\begin{question}
		Given a family of finite groups, what is $\max_G(\specialMinCharacterSetCardinality{G})$, where the maximum ranges over all of the groups in the family?
	\end{question}

	In particular, for the case of $\symmetric{n}$, we prove in \autoref{section:sn} that $\max_n(\specialMinCharacterSetCardinality{\symmetric{n}}) \le 8$ for $n \ge 34$.
	
	\begin{question}
		What is the value of $\max_n(\specialMinCharacterSetCardinality{\symmetric{n}})$?
	\end{question}

	\appendix
	
	\section{Characters and Subgroups of $\symmetric{n}$} \label{section:sn}
	\paragraph{ } In this section, we consider the group $\symmetric{n}$. We use methods similar to those used in this paper to find a set of characters $\specialCharacterSet{\symmetric{n}}$ \textCMarksAllProperSubgroups{that}{$\symmetric{n}$}. Note that in \autoref{example:four_transitive_characters_of_sn}, we presented such a set with $12$ elements, and the set we present in this section consists of $8$ elements.
	
	We introduce notation similar to the notation in the case of $\GL{n}{q}$. Let $\alpha = (\alpha_1, \alpha_2, \dots, \alpha_s)$, such that $\sum_{i=1}^s \alpha_i = n$. We denote the induction of the trivial character of $\symmetric{{\alpha_1}} \times \dots \times \symmetric{{\alpha_s}}$ to $\symmetric{n}$ by $\trivialCharacter{\alpha}$.
	
	\begin{lemma}[e.g. {\cite[Chapter 7.3, Corollary 1]{sn_characters}}] \label{lemma:sn-induced-trivial-character-split}
		Let $\lambda = (\lambda_1, \dots, \lambda_l) \vdash n$. Then
		\[\trivialCharacter{\lambda} = \sum_{\mu \vdash n}K_{\mu \lambda}\snCharacter{\mu},\]
		where $K_{\mu \lambda}$ are the Kostka numbers, and $\snCharacter{\mu}$ are irreducible characters.
	\end{lemma}
	
	We define the characters \defineTextNotation{$\snCharacter{\mu}$}$\in \irreducibleCharacters{\symmetric{n}}$ as these irreducible summands of $\trivialCharacter{\lambda}$. Unlike the case of $\GL{n}{q}$, these are all the irreducible characters of $\symmetric{n}$. This definition agrees with the usual correspondence between partitions and irreducible characters (e.g. \cite{sn_characters}).
	
	We begin by considering the character $\trivialCharacter{(n-3, 1^3)}$ of $\symmetric{n}$. It is the character induced by the permutation representation of $\symmetric{n}$ on $3$-tuples of elements of $X_n = \{1,\dots,n\}$. Therefore, by \autoref{corollary:permutation-character-marks-transitive}, $\trivialCharacter{(n-3, 1^3)} - \trivialCharacter{(n)}$ marks every subgroup of $\symmetric{n}$ that does not act on $X_n$ $3$-transitively. By decomposing the character, we arrive at the following:
	
	\begin{proposition} \label{proposition:sn_not_3_transitive}
		Let $n \ge 6$ and let 
		\[\specialCharacterSet{\symmetric{n}} =
		\begin{Bmatrix}
		\snCharacter{(n-1, 1)} & \snCharacter{(n-2, 2)} & \snCharacter{(n-2, 1^2)} \\ \snCharacter{(n-3, 3)} & \snCharacter{(n-3, 2, 1)} & \snCharacter{(n-3, 1^3)}
		\end{Bmatrix} \subseteq \irreducibleCharacters{\symmetric{n}}.\]
		Then, \textCMarksSubgroupsThatAreNotkTransitive{$\specialCharacterSet{\symmetric{n}}$}{$\symmetric{n}$}{3} under the usual action.
	\end{proposition}

	Therefore, it remains to consider the $3$-transitive subgroups of $\symmetric{n}$. By the classification of of finite multiply transitive groups, we arrive at the following:
	
	\begin{proposition} \label{proposition:3_transitive_groups}
		Let $n \ge 25$. Then, the only possibilities for $3$-transitive subgroups of $\symmetric{n}$, along with the $3$-transitive actions, are:
		\begin{enumerate}
			\item\label{proposition:3_transitive_groups1} The groups $\symmetric{n}$ and $\alternating{n}$ with the usual action on $\{1, \dots, n\}$.
			\item\label{proposition:3_transitive_groups2} Groups $\PSL{2}{q} \subgroup G \subgroup \PGammaL{2}{q}$ with the usual action on $\projective^{1} \field{q}$, for $q = n-1 $ a prime power. Such a group $G$ is $3$-transitive if and only if\declareFootnote{3_transitive_examples}{For example, $\PSL{2}{q}$ is not $3$-transitive when $q$ is odd, while $\PGL{2}{q}$ is always $3$-transitive.} for every $A \in \PGL{2}{q}$ there exists some $\sigma \in \aut{\field{q}}$ such that $(A, \sigma) \in G$.
			\item\label{proposition:3_transitive_groups3} The affine linear groups $\AGL{d}{2}$ with the usual action on $\field{2}^d$, for $n = 2 ^ d$ with $d \ge 5$.
		\end{enumerate}
	\end{proposition}
	\begin{proof}
		The possibilities for $2$-transitive groups can be seen in, e.g. \cite[Chapter 7.7]{permutation_groups}. The only $3$-transitive groups are either $\symmetric{n}$ or $\alternating{n}$ with their usual action (listed in (\ref{proposition:3_transitive_groups1})), groups $\PSL{2}{q} \subgroup G \subgroup \PGammaL{2}{q}$ with the usual action on $\projective^{1} \field{q}$, for $q = n-1 $ a prime power, or subgroups of $\AGammaL{d}{q}$ with the action on $\field{q}^d$. This can be seen in, e.g. \cite[Theorem 5.3]{permutation_groups_2_transitive}.
		
		\paragraph{$\mathbf{\PSL{2}{q} \subgroup G \subgroup \PGammaL{2}{q}}$}
		If $q = n-1 $ is a prime power, a group ${\PSL{2}{q} \subgroup G \subgroup \PGammaL{2}{q}}$ with the usual action on $\projective^{1} \field{q}$ is $3$-transitive if and only if given some distinct $x_1, x_2, x_3 \in \projective^{1} \field{q}$, for every distinct $y_1, y_2, y_3 \in \projective^{1} \field{q}$ there exists some $g \in G$ such that $gx_i = y_i$ for $i = 1, 2, 3$. In particular, denote by $q_0$ the characteristic of $\field{q}$. Then $\field{q_0}$ is a subfield of $\field{q}$ fixed by the automorphisms in $\aut{\field{q}}$.
		
		In particular, if we choose distinct ${x_1, x_2, x_3 \in \projective^{1} \field{q_0}}$, we get ${(A, \sigma)(x_i) = Ax_i}$ for every $(A, \sigma) \in G \subgroup \PGammaL{2}{q}$. Then, $G$ is $3$-transitive if and only if for every distinct $y_1, y_2, y_3 \in \projective^{1} \field{q}$ there exists some $(A,\sigma) \in G$ such that $Ax_i = y_i$ for $i = 1, 2, 3$.
		
		Since $\PGL{2}{q}$ acts $3$-transitively on $\projective^{1} \field{q}$ and is of order $q(q^2-1)$, it is sharply $3$-transitive. Therefore, there is a bijection between distinct ${y_1, y_2, y_3 \in \projective^{1} \field{q}}$ and ${A \in \PGL{2}{q}}$ such that $Ax_i = y_i$ for $i = 1, 2, 3$. We denote ${A = A(y_1, y_2, y_3)}$.
		
		To sum up, the following are equivalent:
		\begin{enumerate}
			\item $G$ is $3$-transitive.
			\item For every distinct $y_1, y_2, y_3 \in \projective^{1} \field{q}$ there exists some $(A,\sigma) \in G$ such that $Ax_i = y_i$ for $i = 1, 2, 3$.
			\item For every distinct $y_1, y_2, y_3 \in \projective^{1} \field{q}$ there exists some $\sigma \in \aut{\field{q}}$ such that $(A(y_1, y_2, y_3), \sigma) \in G$.
			\item For every $A \in \PGL{2}{q}$ there exists some $\sigma \in \aut{\field{q}}$ such that $(A, \sigma) \in G$.
		\end{enumerate}
		
		\paragraph{$\mathbf{G \subgroup \AGammaL{d}{q}}$}
		The only case remaining is where $G \subgroup \AGammaL{d}{q}$. Assume $G$ is $3$-transitive and consider the stabilizer $\stab{G}{0} \subgroup \GammaL{d}{q}$, which acts $2$-transitively on $\field{q}^d \setminus \{0\}$. We use linearity to claim this implies that $q = 2$. Denote by $q_0$ the characteristic of $\field{q}$, and let $\ell$ be the exponent, $q = q_0 ^ \ell$.
		\begin{description}[leftmargin=2cm]
			\item[$\bm{d \ge 2, q \ge 3}$] There is no $g \in \stab{G}{0}$ with $g(e_1, a e_1) = (e_1, e_2)$, for $a \in \field{q}, a \ne 0, 1$.
			\item[$\bm{d = 1}$] Let $q = q_0^\ell$ with $q_0$ prime. Observe that ${25 \le n = q}$. Let ${a \in \field{q}}$, ${a \ne 0, 1}$. Consider the orbit $\orbit{\aut{\field{q}}}{a}$, which is of size at most $\ell$. Therefore, $\cardinality{\field{q} \setminus (\{0, 1\} \cup \orbit{\aut{\field{q}}}{a})} = q - \ell - 2$. It is easy to see that $q \ge \ell + 3$ for $q \ge 25$. As such, we can choose $b$ such that $b \in \field{q}$, $b \ne 0, 1$ and $b$ is not in the same orbit of $\aut{\field{q}}$ as $a$. Then, there is no $g \in \stab{G}{0}$ with ${g(e_1, a e_1) = (e_1, b e_1)}$.
		\end{description}
		
		Therefore, $q = 2$. Since $n \ge 25$, we get that $d \ge 5$. As such, $\stab{G}{0}$ acts $2$-transitively on ${\field{2}^d \setminus \{0\} = \projective^{d-1} \field{2}}$. By \autoref{theorem:2_transitive_collineation_groups}, ${\stab{G}{0} \supergroup \SL{d}{2} = \GL{d}{2} = \GammaL{d}{2}}$. Therefore, $\stab{G}{0} = \GL{d}{2}$ and ${G = \AGL{d}{2}}$.
	\end{proof}

	\begin{proposition} \label{proposition:3_transitive_groups_chi}
		Let $n \ge 25$ and $n \ne 33$. Then, ${\trivialCharacter{(n-4, 4)} - \trivialCharacter{(n)}}$ marks every $3$-transitive proper subgroup of $\symmetric{n}$, except $\alternating{n}$.
	\end{proposition}
	\begin{proof}
		Let $K$ be one of the groups listed in \autoref{proposition:3_transitive_groups_chi}. Note that $\trivialCharacter{(n-4, 4)}$ is the character of the permutation representation of $\symmetric{n}$ acting on $Y_n$, the collection subsets of $X_n = \{1, \dots, n\}$ of size $4$.
		
		By \autoref{corollary:permutation-character-marks-transitive}, $\trivialCharacter{(n-4, 4)}$ marks precisely the subgroups of $\symmetric{n}$ that act transitively on $Y_n$. Thus, it is sufficient to show that $K$ does not act transitively on $Y_n$.
		
 		\begin{description}[leftmargin=2cm]
		 	\item[$\bm{K \subgroup \PGammaL{2}{q}}$] Let $q = q_0^\ell$, with $q_0$ prime. By \eqref{corollary:orbit_stabilizer_orders}, $\cardinality{Y_n} \divides \cardinality{K}$. Note that $\cardinality{Y_n} = \binom{n}{4}$. Since $\cardinality{K} \divides \cardinality{\PGammaL{2}{q}}$, we get ${\binom{q+1}{4} = \cardinality{Y_{q+1}} \divides \cardinality{\PGammaL{2}{q}} = \ell (q^3-q)}$, which is equivalent to ${(q-2) \divides 24\ell}$. It is easy to see that no ${q = n-1 \ge 24}$ except $q = 32$ (which is explicitly omitted) satisfies this equation.
		 	\item[$\bm{K = \AGL{d}{2}}$] We know that $d \ge 5$. It is easy to see that if $x_1, x_2, x_3, x_4$ satisfy $x_1 + x_2 = x_3 + x_4$ then $gx_1 + gx_2 = gx_3 + gx_4$ for every $g \in K$. As such, there exists no $g \in K$ such that ${g \{0, e_1, e_2, e_1 + e_2\} = \{0, e_1, e_2, e_3\}}$.
		 \end{description}
	\end{proof}

	We use \autoref{proposition:sn_not_3_transitive} and \autoref{proposition:3_transitive_groups_chi}. Then, we decompose the character ${\trivialCharacter{(n-4, 4)} - \trivialCharacter{(n)} + \trivialCharacter{(n-3, 1^3)} - \trivialCharacter{(n)}}$ to irreducible characters and add the sign character (which is equal to $\snCharacter{(1^n)}$) to arrive at the following:
	
	\begin{corollary}
		Let $n \ge 34$ and let 
		\[\specialCharacterSet{\symmetric{n}} = 
		\begin{Bmatrix}
		\snCharacter{(n-1, 1)} & \snCharacter{(n-2, 2)} & \snCharacter{(n-2, 1^2)} & \snCharacter{(n-3, 3)} \\
		\snCharacter{(n-3, 2, 1)} & \snCharacter{(n-3, 1^3)} & \snCharacter{(n-4, 4)} & \snCharacter{(1^n)}
		\end{Bmatrix} \subseteq \irreducibleCharacters{\symmetric{n}}.\]
		Then, \textCMarksAllProperSubgroups{$\specialCharacterSet{\symmetric{n}}$}{$\symmetric{n}$}.
	\end{corollary}
	
	\section{Group Extensions} \label{section:group_extensions}
	\paragraph{ } Throughout this section, given a finite group $H$, we denote by $\specialMinCharacterSetCardinality{H}$ the minimal cardinality of a subset of nontrivial irreducible characters \textCMarksProperSubgroups{that}{$H$}. To simplify the definitions, we say that a set \defineTextNotation{$\specialCharacterSet{H}$ \textit{achieves} $\specialMinCharacterSetCardinality{H}$} if $\specialCharacterSet{H}$ is a subset of nontrivial irreducible characters of $H$ \textCMarksProperSubgroups{that}{$H$} with $\cardinality{\specialCharacterSet{H}} = \specialMinCharacterSetCardinality{H}$.
	
	Let $G$ be a finite group and let $N \normalSubgroup G$. We study the relations between $\specialMinCharacterSetCardinality{G}$, $\specialMinCharacterSetCardinality{\factor{G}{N}}$ and $\specialMinCharacterSetCardinality{N}$.
	
	\subsection{$\specialMinCharacterSetCardinality{G}$ and $\specialMinCharacterSetCardinality{\factor{G}{N}}$}
	\paragraph{ } Denote by $f \colon G \twoheadrightarrow \factor{G}{N}$ the quotient epimorphism. If $\chi \in \characters{\factor{G}{N}}$, the character can be pulled back by $f$ to obtain a character $\pullbackChiByF{\chi}{f}$ of $G$, ${\pullbackChiByF{\chi}{f}(x) = \chi(f(x))}$. It is easy to see that $\chi$ is irreducible if and only if $\pullbackChiByF{\chi}{f}$ is irreducible, and that ${\pullbackSymbol{f} \colon \characters{\factor{G}{N}} \to \characters{G}}$ is injective.
	
	For every $K \subgroup G$ such that $K \supergroup N$, we denote the quotient epimorphism $K \twoheadrightarrow \factor{K}{N}$ by $f_K$, and the pull-back from characters of $\factor{K}{N}$ to characters of $K$ by $\pullbackSymbol{f_K}$.
	
	The following properties of the pull-back of an epimorphism are easy to verify:
	
	\begin{lemma} \label{lemma:basic_properties_of_pullback}
		Let $G$ be a finite group and let $N \normalSubgroup G$. Let $f \colon G \twoheadrightarrow \factor{G}{N}$ be the quotient epimorphism, let $\phi, \psi \in \characters{\factor{G}{N}}$, let $\theta \in \characters{\factor{K}{N}}$, and let $N \subgroup K \subgroup G$. Then:
		\begin{enumerate}
			\item $\trivialCharacter{G} = \pullbackChiByF{\trivialCharacter{\factor{G}{N}}}{f}$
			\item $\pullbackChiByF{(\phi + \psi)}{f} = \pullbackChiByF{\phi}{f} + \pullbackChiByF{\psi}{f}$
			\item $\innerProduct{\phi}{\psi}_{\factor{G}{N}} = \innerProduct{\pullbackChiByF{\phi}{f}}{\pullbackChiByF{\psi}{f}}_{G}$
			\item $\pullbackChiByF{\induced{\theta}{\factor{K}{N}}{\factor{G}{N}}}{f} = \induced{\pullbackChiByF{\theta}{f_K}}{K}{G}$
		\end{enumerate}
	\end{lemma}
	
	\begin{proposition} \label{proposition:supergroups_of_n_nonvanishing}
		Let $G$ be a finite group, and let $N \normalSubgroup G$. Denote by ${f \colon G \twoheadrightarrow \factor{G}{N}}$ the quotient epimorphism. Let $K \subgroup G$ be some subgroup such that $K \supergroup N$. Let $\chi \in \irreducibleCharacters{G}$ be an irreducible character that satisfies \textChiDoesntVanishOnK{\chi}{K}. Then $\chi \in \Ima(\pullbackSymbol{f})$.
	\end{proposition}
	\begin{proof}
		Since $\chiKContainsTrivial{\chi}{K}$, using \nameref{frobenius_reciprocity} we deduce that
		\[0 < \characterSum{\chi}{K} = \innerProduct{\chi}{\inducedTrivial{K}{G}}_G\]
		which implies that $\phiContainsPsi{\inducedTrivial{K}{G}}{\chi}$.
		
		Let $\inducedTrivial{\factor{K}{N}}{\factor{G}{N}} = \sum_{i}{\phi_i}$ for some $\phi_i \in \irreducibleCharacters{\factor{G}{N}}$, possibly with repetitions. We have $\inducedTrivial{K}{G} = \induced{\pullbackChiByF{\trivialCharacter{\factor{K}{N}}}{f_K}}{K}{G} = \pullbackChiByF{\inducedTrivial{\factor{K}{N}}{\factor{G}{N}}}{f} = \sum_{i}{\pullbackChiByF{\phi_i}{f}}$, which implies that all irreducible summands of $\inducedTrivial{K}{G}$ are in $\Ima(\pullbackSymbol{f})$. In particular, $\chi \in \Ima(\pullbackSymbol{f})$.
	\end{proof}
	
	\begin{proposition} \label{proposition:characters_of_g_to_g_mod_n}
		Let $G$ be a finite group, and let $N \normalSubgroup G$. Denote by ${f \colon G \twoheadrightarrow \factor{G}{N}}$ the quotient epimorphism. Let $\specialCharacterSet{G} \subseteq \irreducibleCharacters{G}$ be some set of nontrivial irreducible characters \textCMarksProperSubgroups{that}{$G$}, and let 
		\[\specialCharacterSet{\factor{G}{N}} \define \left\{\chi \in \irreducibleCharacters{\factor{G}{N}} \colon \pullbackChiByF{\chi}{f} \in \specialCharacterSet{G}\right\}.\]
		Then, \textCMarksProperSubgroups{$\specialCharacterSet{\factor{G}{N}}$}{$\factor{G}{N}$}, and consists of nontrivial irreducible characters.
	\end{proposition}
	\begin{proof}
		Let $J \subgroup \factor{G}{N}$. By the correspondence theorem, there exists some $K \subgroup G$ with $K \supergroup N$ and $f(K) = J$. Thus, there exists some $\chi \in \specialCharacterSet{G}$ such that $\chiKContainsTrivial{\chi}{K}$.
		
		By \autoref{proposition:supergroups_of_n_nonvanishing}, $\chi \in \Ima(\pullbackSymbol{f})$. Therefore, there exists some $\chi' \in \irreducibleCharacters{\factor{G}{N}}$ such that $\pullbackChiByF{\chi'}{f} = \chi$. This implies that $\chi' \in \specialCharacterSet{\factor{G}{N}}$.
		
		By \autoref{lemma:basic_properties_of_pullback} (with $G=K$), $\characterSum{\chi'}{J} = \characterSum{\chi}{K} > 0$.
		
		To prove the characters in $\specialCharacterSet{\factor{G}{N}}$ are nontrivial and irreducible, let ${\chi \in \specialCharacterSet{\factor{G}{N}}}$. Since ${\pullbackChiByF{\chi}{f} \in \specialCharacterSet{G}}$, it is nontrivial and irreducible. By the properties of pull-backs, so is $\chi$.
	\end{proof}

	\paragraph{Remark} This definition of $\specialCharacterSet{\factor{G}{N}}$ is the same definition as the one made in \autoref{section:projective_unipotent_restriction}.

	\begin{corollary} \label{corollary:XG_ge_XG_N}
		\[\specialMinCharacterSetCardinality{\factor{G}{N}} \le \specialMinCharacterSetCardinality{G}.\]
	\end{corollary}
	\begin{proof}
		Let $\specialCharacterSet{G} \subseteq \irreducibleCharacters{G}$ be a set achieving $\specialMinCharacterSetCardinality{G}$, let ${f \colon G \twoheadrightarrow \factor{G}{N}}$ be the quotient epimorphism, and let $\specialCharacterSet{\factor{G}{N}}$ be as defined in \autoref{proposition:characters_of_g_to_g_mod_n}. We get an injection \[{\restricted{\pullbackSymbol{f}}{\specialCharacterSet{\factor{G}{N}}} \colon \specialCharacterSet{\factor{G}{N}} \to \specialCharacterSet{G}},\]
		which implies
		\[{\specialMinCharacterSetCardinality{\factor{G}{N}} \le  \cardinality{\specialCharacterSet{\factor{G}{N}}} \le \cardinality{\specialCharacterSet{G}} = \specialMinCharacterSetCardinality{G}}.\]
	\end{proof}

	\begin{corollary} \label{corollary:X_contains_character_for_each_normal_subgroup}
		Let $G$ be a finite group, and let $\specialCharacterSet{G} \subseteq \irreducibleCharacters{G}$ be some set of nontrivial irreducible characters \textCMarksProperSubgroups{that}{$G$}. Then, for every $N \normalSubgroup G$ there exist $\specialMinCharacterSetCardinality{\factor{G}{N}}$ many characters $\chi \in \specialCharacterSet{G}$ such that ${\restricted{\chi}{N} = \dim(\chi)\trivialCharacter{N}}$.
		
		In particular, if $N \normalSubgroup G$ is proper, there exists at least one such character.
	\end{corollary}
	\begin{proof}
		Let $N \normalSubgroup G$, denote by $f \colon G \to \factor{G}{N}$ the quotient epimorphism, and let $\unit_{\factor{G}{N}}$ be the identity element of $\factor{G}{N}$. We define $\specialCharacterSet{\factor{G}{N}}$ as in \autoref{proposition:characters_of_g_to_g_mod_n}. Then, ${\specialMinCharacterSetCardinality{\factor{G}{N}} \le  \cardinality{\specialCharacterSet{\factor{G}{N}}}}$ and every character in $\specialCharacterSet{\factor{G}{N}}$ is mapped to a different character in $\specialCharacterSet{G}$ by the injection $\pullbackSymbol{f}$. For every $\chi \in \specialCharacterSet{\factor{G}{N}}$, the pull-back $\pullbackChiByF{\chi}{f} \in \specialCharacterSet{G}$ satisfies
		\[\pullbackChiByF{\chi}{f}(g) = \chi(f(g)) = \chi\left(\unit_{\factor{G}{N}}\right) = \dim(\chi),\]
		for every $g \in N$. Thus, the claim follows.
		
		If $N \normalSubgroup G$ is proper, $\factor{G}{N}$ is not the trivial group and thus has the group $\left\{\unit_{\factor{G}{N}}\right\}$ as a proper subgroup. As such, $1 \le \specialMinCharacterSetCardinality{\factor{G}{N}}$, as every $\specialCharacterSet{\factor{G}{N}}$ achieving it must contain at least one character to mark this proper subgroup.
	\end{proof}

	\subsection{Abelian Groups}
	\paragraph{ } In this section, we consider $G$ which are finite abelian groups. These provide useful examples for \autoref{section:XG_and_XN}.
	
	\begin{proposition} \label{corollary:abelian_group_maximal_marking}
		Let $G$ be a finite abelian group, and let $\specialCharacterSet{G} \subseteq \irreducibleCharacters{G}$ be a set achieving $\specialMinCharacterSetCardinality{G}$ (since $G$ is abelian, its elements are homomorphisms). Then, there is a bijection
		\[\left\{ H \subgroup G \colon H \text{ is maximal}\right\} \leftrightarrow \specialCharacterSet{G},
		H \mapsto \chi_H
		\]
		where $\restricted{\chi_H}{H} = \trivialCharacter{H}$ and $H = \ker(\chi_H)$.
	\end{proposition}
	\begin{proof}
		Since $G$ is abelian, every subgroup of $G$ is normal. Thus, by \autoref{corollary:X_contains_character_for_each_normal_subgroup}, we can choose a character $\chi_H \in \specialCharacterSet{G}$ for every maximal subgroup $H$ that satisfies $\restricted{\chi_H}{H} = \dim(\chi_H)\trivialCharacter{H}$. Since every character of an abelian group is linear, $\dim(\chi_H) = 1$ and $H \le \ker(\chi_H)$. By the maximality of $H$ and nontriviality of $\chi_H$,  $H = \ker(\chi_H)$. It remains to show that this choice is a bijection.
		
		\begin{description}[leftmargin=1.1cm]
			\item[One-to-one] Let $H, H' \subgroup G$ be maximal subgroups with $\chi_H = \chi_{H'}$. Taking the kernels of these characters, we get $H = H'$.
			
			\item[Onto] The set $\{\chi_H \colon H \text{ is maximal}\}$ marks every maximal subgroup of $G$, and as such, it marks every proper subgroup of $G$ (by the same argument as the one made in the beginning of \autoref{section:maximal_subgroups_of_glnq}). By the minimality of ${\specialMinCharacterSetCardinality{G} = \cardinality{\specialCharacterSet{G}}}$, the mapping is onto.
		\end{description}
	\end{proof}

	\paragraph{Cyclic groups} Let $\cyclic{n} = \factor{\integers}{n\integers}$ be the cyclic group of order $n$, for $n \in \positiveIntegers$.
	
	\begin{proposition} \label{proposition:XCn_cardinality}
		For every $n \in \positiveIntegers$, $\specialMinCharacterSetCardinality{\cyclic{n}}$ is precisely the number of different prime divisors of $n$.
	\end{proposition}
	\begin{proof}
		Let $g \in \cyclic{n}$ be a generator. It is known that the subgroups of $\cyclic{n}$ are precisely $\{ \groupSpan{g^k} \colon k \divides n\}$, and all of these groups are different. It is easy to see that $\groupSpan{g^k}$ is maximal if and only if $k$ is prime. Thus, by \autoref{corollary:abelian_group_maximal_marking}, ${\specialMinCharacterSetCardinality{\cyclic{n}} = \cardinality{\{p \in \positiveIntegers \colon p \text{ is prime}, p \divides n\}}}$
	\end{proof}

	\begin{corollary}
		There exist infinitely many finite groups $G$ with $N \normalSubgroup G$ such that ${\specialMinCharacterSetCardinality{G} = a \specialMinCharacterSetCardinality{\factor{G}{N}}}$, for every $a \in \positiveIntegers$.
	\end{corollary}
	\begin{proof}
		Let $G = \cyclic{2^m n}$, where $n$ is odd, and has $a-1$ distinct prime factors. For $N = \cyclic{n}$, the quotient satisfies $\factor{G}{N} = \cyclic{2^m}$. Then, ${\specialMinCharacterSetCardinality{\cyclic{2^m n}} = a = a \specialMinCharacterSetCardinality{\cyclic{2^m}}}$.
	\end{proof}

	\begin{corollary} \label{corollary:XGLn_lower_bound}
		Let $n \in \positiveIntegers$. Then, $\specialMinCharacterSetCardinality{\GL{n}{q}}$ is at least the number of prime divisors of $q-1$.
	\end{corollary}
	\begin{proof}
		By \autoref{corollary:XG_ge_XG_N} with $N = \SL{n}{q}$, $\specialMinCharacterSetCardinality{\GL{n}{q}} \ge \specialMinCharacterSetCardinality{\cyclic{q-1}}$. The claim follows by \autoref{proposition:XCn_cardinality}.
	\end{proof}

	\medskip

	\paragraph{Vector spaces} We continue with the additive groups of the vector spaces $\field{p}^\ell$, where $p$ is prime and $\ell \in \positiveIntegers$.

	\begin{proposition} \label{proposition:XFpd_cardinality}
		Let $p$ be prime and let $\ell \in \positiveIntegers$. Then, $\specialMinCharacterSetCardinality{\field{p}^\ell} = \frac{p^\ell - 1}{p - 1}$.
	\end{proposition}
	\begin{proof}
		Denote by $f_w(x) = \sum_{i = 1}^{\ell}x_i w_i$. A maximal subgroup of $\field{p}^\ell$ is a $(\ell - 1)$-dimensional subspace of $\field{p}^\ell$, namely, some ${W_{\projectiveVector{w}} = \{x \in \field{p}^\ell \colon f_w(x) = 0\} = \ker f_w}$ where $0 \ne w \in \field{p}^\ell$ and $\projectiveVector{w} \in \projective^{\ell-1}\field{p}$ is the point in the projective space represented by $w$. This shows a bijection between the maximal subgroups of $\field{p}^\ell$ and $\projective^{\ell-1}\field{p}$. Thus, \[\specialMinCharacterSetCardinality{\field{p}^\ell} = \cardinality{\projective^{\ell-1}\field{p}} = \frac{p^\ell - 1}{p - 1}.\]
	\end{proof}

	\subsection{$\specialMinCharacterSetCardinality{G}$ and $\specialMinCharacterSetCardinality{N}$} \label{section:XG_and_XN}
	\paragraph{ } In this section, we study the relation between $\specialMinCharacterSetCardinality{G}$ and $\specialMinCharacterSetCardinality{N}$. We begin with the following proposition.
	
	\begin{proposition} \label{proposition:XG_from_XN_XG_N}
		For a finite group $H$, let $\specialCharacterSet{H}$ be a set achieving $\specialMinCharacterSetCardinality{H}$. Let $G$ be a finite group, let $N \normalSubgroup G$, and let $f \colon G \twoheadrightarrow \factor{G}{N}$ be the quotient epimorphism. Define
		\[\overline{\specialCharacterSet{G}} \define \{\induced{\chi}{N}{G} \colon \chi \in \specialCharacterSet{N}\} \cup \{\pullbackChiByF{\phi}{f} \colon \phi \in \specialCharacterSet{\factor{G}{N}}\}.\]
		
		Then, \textCMarksProperSubgroups{$\overline{\specialCharacterSet{G}}$}{$G$}.
	\end{proposition}
	\begin{proof}
		Let $K \subgroup G$ and let $\chi \in \specialCharacterSet{N}$. Denote by $\chi^{(t)}$ the character of $N$ defined by $\chi^{(t)}(g) = \chi(t^{-1}gt)$, for every $t \in G$.
		
		\begin{multline*}
			\begin{aligned}
				\characterSum{\induced{\chi}{N}{G}}{K}
				&= \frac{1}{\cardinality{K}}\sum_{k \in K} \induced{\chi}{N}{G}(k) \\
				&= \frac{1}{\cardinality{K} \cdot \cardinality{N}}\sum_{k \in K} \left(\sum_{t \in G \mathSuchThat t^{-1}kt \in N} \chi(t^{-1}kt)\right) \\
				&= \frac{1}{\cardinality{K} \cdot \cardinality{N}}\sum_{t \in G} \left(\sum_{k \in K \mathSuchThat t^{-1}kt \in N} \chi(t^{-1}kt)\right) \\
				&= \frac{1}{\cardinality{K} \cdot \cardinality{N}}\sum_{t \in G} \sum_{k \in K \cap N} \chi(t^{-1}kt) \\
				&= \frac{\cardinality{N \cap K}}{\cardinality{K} \cdot \cardinality{N}}\sum_{t \in G} \characterSum{\chi^{(t)}}{K \cap N} \\
				& \ge \frac{\cardinality{N \cap K}}{\cardinality{K} \cdot \cardinality{N}} \characterSum{\chi}{K \cap N}.
			\end{aligned}
		\end{multline*}
		
		In particular, if $K \cap N \ne N$ there exists some $\chi \in \specialCharacterSet{N}$ such that $\characterSum{\induced{\chi}{N}{G}}{K} > 0$.
		
		Therefore, \textCMarksSubgroupsThatDoNotContainH{$\overline{\specialCharacterSet{G}}$}{$G$}{$N$}, and we may assume that $K \supergroup N$.
		
		In this case, there exists some $\phi \in \specialCharacterSet{\factor{G}{N}}$ such that \textChiDoesntVanishOnK{\phi}{\factor{K}{N}}. Then, $\characterSum{\pullbackChiByF{\phi}{f}}{K} = \characterSum{\phi}{\factor{K}{N}} > 0$, which completes the proof. 
	\end{proof}

	\medskip
	
	The problem with this proposition is that the set $\overline{\specialCharacterSet{G}}$ may not consist of irreducible characters. Therefore, we decompose the characters of $\overline{\specialCharacterSet{G}}$ to get a bound on $\specialMinCharacterSetCardinality{G}$.

	\begin{corollary} \label{corollary:XG_inequality_XN_XG_N}
		Let $G$ be a finite group and let $N \normalSubgroup G$. Then
		
		\[\specialMinCharacterSetCardinality{\factor{G}{N}} \le \specialMinCharacterSetCardinality{G} \le \subgroupIndex{G}{N} \specialMinCharacterSetCardinality{N} + \specialMinCharacterSetCardinality{\factor{G}{N}}\]
	\end{corollary}
	\begin{proof}
		The left inequality follows from \autoref{corollary:XG_ge_XG_N}. To prove the right inequality it is sufficient to show, by \autoref{proposition:XG_from_XN_XG_N}, that for every $\chi \in \irreducibleCharacters{N}$ the induced character $\induced{\chi}{N}{G}$ decomposes to at most $\subgroupIndex{G}{N}$ irreducible characters of $G$.
		
		Let $\chi \in \irreducibleCharacters{N}$, and let $\phi$ be an irreducible summand of $\induced{\chi}{N}{G}$. By \nameref{frobenius_reciprocity}, $\innerProduct{\restricted{\phi}{N}}{\chi}_N = \innerProduct{\phi}{\induced{\chi}{N}{G}}_G > 0$. Therefore $\phiContainsPsi{\restricted{\phi}{N}}{\chi}$, which implies that $\dim(\phi) = \dim(\restricted{\phi}{N}) \ge \dim(\chi)$.
		
		Assume that $\induced{\chi}{N}{G}$ decomposes to $k$ irreducible characters of $G$, which we denote by $\phi_1, \dots, \phi_k$ (these may be not distinct). By calculating the degrees:		
		\[\subgroupIndex{G}{N}\dim(\chi) = \dim(\induced{\chi}{N}{G}) = \dim\left(\sum_{i = 1}^{k}\phi_i\right) = \sum_{i = 1}^{k}\dim(\phi_i) \ge k\dim(\chi)\]
		which implies that $k \le \subgroupIndex{G}{N}$.
	\end{proof}
	
	We show these inequalities are sharp by considering two examples.
	
	\begin{example}
		Let $G = \cyclic{p ^ \ell}$ and $N = \cyclic{p ^ {\ell-1}} \normalSubgroup G$. By \autoref{proposition:XCn_cardinality}, every $\specialCharacterSet{G}$ achieving $\specialMinCharacterSetCardinality{G}$ consists of a single character, which we denote by $\chi$. By the properties of $\specialCharacterSet{G}$, when restricting $\chi$ to $N$, we get ${\restricted{\chi}{N} = \trivialCharacter{N}}$, which cannot be a member of $\specialCharacterSet{N}$. Then, by \nameref{frobenius_reciprocity}, $\chi$ cannot be the induction of a nontrivial irreducible character of $N$. In this case, ${\specialMinCharacterSetCardinality{\factor{G}{N}} = \specialMinCharacterSetCardinality{G}}$.
	\end{example}
	
	\begin{example}
		Let $p$ be prime and let $\ell \in \positiveIntegers$. Let $G = \field{p}^\ell$ and $N \normalSubgroup G$. We use standard facts about $G$ and \autoref{proposition:XFpd_cardinality}.
		\begin{enumerate}
			\item $G \cong \field{p} ^ \ell$, $\specialMinCharacterSetCardinality{G} = \frac{p^\ell - 1}{p - 1}$
			\item $N \cong \field{p} ^ i$ for some $0 \le i \le \ell$, $\specialMinCharacterSetCardinality{N} = \frac{p^i - 1}{p - 1}$
			\item $\factor{G}{N} \cong \field{p} ^ {\ell - i}$, $\specialMinCharacterSetCardinality{\factor{G}{N}} = \frac{p^{\ell - i} - 1}{p - 1}$
		\end{enumerate}
	
		From these we deduce: $\specialMinCharacterSetCardinality{G} = \subgroupIndex{G}{N} \specialMinCharacterSetCardinality{N} + \specialMinCharacterSetCardinality{\factor{G}{N}}$.
	\end{example}

	\section*{Acknowledgments}
	This paper forms part of an M.Sc.~Thesis written by the author under the supervision of Dr. Doron Puder from Tel Aviv University. I deeply thank my thesis supervisor for introducing me to many fascinating topics in mathematics and showing me the beauty in them, as well as for his advice, guidance, patient readings of the various drafts and the time he has invested in the process of making this paper.
	I thank Professor David Soudry from Tel Aviv University, for sharing his knowledge about the general linear group.
	The research was supported by the Israel Science Foundation, ISF grant 1071/16.

	\bibliographystyle{alpha}
	\bibliography{paper}
\end{document}